\documentclass[12pt]{article}
\usepackage[a4paper, total={17cm,25cm}]{geometry}
\usepackage{amsmath}
\usepackage{amsmath,amssymb,amsfonts,amsthm}
\usepackage{color, graphics}
\usepackage{multicol}

\usepackage{multirow}
\usepackage{tikz}
\usetikzlibrary{arrows,patterns,calc}

\allowdisplaybreaks

\usepackage{float}
\usepackage{hyperref}

\usepackage[shortlabels]{enumitem}

\graphicspath{ {./ImagesForFigures/} }

\usepackage[alphabetic,y2k,initials]{amsrefs}

\def\mod {\mathrm{mod}}

\newcommand{\E}{\mathcal{E}}
\newcommand{\C}{\mathcal{C}}
\renewcommand{\L}{\mathcal{L}}
\newcommand{\A}{\mathcal{A}}
\newcommand{\D}{\mathcal{D}}

\numberwithin{equation}{section}

\newtheorem{theorem}{Theorem}[section]
\newtheorem{proposition}[theorem]{Proposition}
\newtheorem{corollary}[theorem]{Corollary}

\theoremstyle{definition}%% only works if amsthm package is used
\newtheorem{example}[theorem]{Example}
\newtheorem{definition}[theorem]{Definition}
\newtheorem{remark}[theorem]{Remark}

 \title{Billiard Ordered Games and Books}

\usepackage{authblk}
\author[1,3]{Vladimir Dragovi\'c}
\author[2]{Sean Gasiorek}
\author[2,3]{Milena Radnovi\'c}
\affil[1]{\textsc{The University of Texas at Dallas, Department of Mathematical Sciences}}
\affil[2]{\textsc{The University of Sydney, School of Mathematics and Statistics}}
\affil[3]{\textsc{Mathematical Institute SANU, Belgrade}}
\affil[ ]{\texttt{vladimir.dragovic@utdallas.edu, sean.gasiorek@sydney.edu.au, milena.radnovic@sydney.edu.au}}

\date{}

\begin{document}

\maketitle

\centerline{\it In Memory of Alexey V. Borisov (1965--2021).}

\begin{abstract}
The aim of this work is to put together two novel concepts from the theory of integrable billiards: billiard ordered games  and confocal billiard books. 
Billiard books appeared recently in the work of Fomenko's school, in particular of V.~Vedyushkina. 
These more complex billiard domains are obtained by gluing planar sets  bounded by arcs of confocal conics along common edges. 
Such domains are used in this paper to construct the configuration space for billiard ordered games.
We analyse dynamical and topological properties of the systems obtained in that way.

\smallskip


\smallskip

\end{abstract}

\tableofcontents

\section{Introduction}\label{intro}

Mathematical billiards provide fundamental examples of a broad family of dynamical systems.
The billiard dynamics essentially depends on the shape of the billiard table.
The elliptic billiard \cites{Bir,KT1991, Tab} connects this family of dynamical systems with the notion of integrability, specifically complete Liouville integrability. Through the Liouville-Arnold theorem, the dynamics of all integrable systems are similar once the dynamics is restricted to the invariant Liouville tori \cite{Ar}. As such, describing the organisation and configuration of Liouville tori is a valuable undertaking in the study of integrable systems. 

Fomenko and his school developed a suite of tools to describe the topological classification of integrable systems using what are now known as \emph{Fomenko graphs} and \emph{Fomenko-Zieschang invariants}, see \cite{Fomenko1987,FZ1991}.
The book \cite{BF2004} contains a detailed description of this type of topological classification, including a large list of well-known integrable systems, such as the integrable cases of rigid body motion and geodesic flows on surfaces.
The use of topological tools in the study of integrable billiards was initiated in \cites{DR2009}, see also \cites{DR2010, DR2011}. Further details and applications to other integrable systems can be found in the literature related to billiards \cites{Fokicheva2014, R2015, DR2017, VK2018,FV2019, FV2019a, PRK2020, DGR2021}. For the applications in the broader theory of Hamiltonian systems with two degrees of freedom see \cites{BMF1990,RRK2008,BBM2010}.

This work is inspired by the recent results of Fomenko's school, in particular V.~Vedyushkina, who introduced more complex billiard domains, obtained by gluing planar sets bounded by arcs of confocal conics along common edges.
By constructing such domains, one gains the flexibility to model various integrable behaviours, with the ultimate goal to show that each integrable system with two degrees of freedom is equivalent to one such a billiard.
See \cites{FVZ2021,Ved2021, VK2020, Ved2020, Khar2020} and references therein for the most recent progress in that direction of research.

In this paper, we use the idea of billiard books in order to get a deeper insight into the billiard ordered games, which were introduced in \cite{DR2004}, see also \cites{DR2006, DR2010, DR2011}.
This approach enabled us to extend the dynamics of the billiard ordered games to any initial conditions along with the study of their topological properties.

This paper is organised as follows.
In Section \ref{sec:billiards}, we recall the basic definitions of elliptic billiards and billiard ordered games, and introduce the billiard books that will be used further.
Section \ref{sec:Examples} consists of the examples where we constructed the billiard books realising certain billiard ordered games. We also analyse the dynamical behaviour of the billiard particle for various caustics and initial conditions. 
In Section \ref{sec:construction}, we provide a general construction of a billiard book where a given billiard ordered game is realised.
Section \ref{sec:topology} contains the topological analysis of the billiards within books associated with the billiard ordered games.

Let us conclude the introduction by observing that both billiard systems and topological analysis of integrable mechanical systems belonged to a wide spectrum of the scientific interests of Alexey Vladimirovich Borisov, see for example \cite{BBM2010} and \cite{BKM2011}. His premature death has been a big loss for our scientific community. His multiple talents, energy, legacy of numerous results, and strong scientific school will remain as the inspiration for generations to come.

\section{Elliptic billiards and billiard ordered games}\label{sec:billiards}

In Section \ref{sec:elliptic-billiards}, we set the notation for the confocal family and recall the definition of billiard within an ellipse.
In Section \ref{sec:games}, the  definition of billiard ordered games is given, along with basic properties.
In Section \ref{sec:books}, the billiard books that will be used in this paper are introduced.

\subsection{Elliptic billiards}\label{sec:elliptic-billiards}

Consider an ellipse in the Euclidean plane:
$$\E\ :\ \frac{x^2}{a}+\frac{y^2}{b}=1, \quad a>b>0.$$

Billiard within $\E$ is a dynamical system where the particle is moving freely inside $\E$ and obeying the \emph{billiard reflection law} when it hits the boundary: the angles of incidence and reflection are congruent to each other and the speed remains unchanged before and after the reflection.

The confocal family of this ellipse is given by:
\begin{equation}
\C_{\lambda}\ :\ \frac{x^2}{a-\lambda} + \frac{y^2}{b-\lambda}=1.
\label{eq:ConFam}
\end{equation}

The members of the confocal family are ellipses for $\lambda<b$, hyperbolae for $\lambda \in (b,a)$, and degenerate for $\lambda \in \{b,a\}$. 

Each trajectory of the billiard within $\E$ has a \emph{caustic}: a curve touching all lines containing segments of the trajectory.
Moreover, the caustic belongs to the confocal family \eqref{eq:ConFam}.

\subsection{Billiard ordered games}\label{sec:games}

We recall here the definition and main properties of billiard ordered games, which were introduced in \cite{DR2004}.
In that work, a generalisation in an arbitrary-dimensional Euclidean space was considered.
Since in this paper we focus on the planar case, all considerations here will be restricted to that.

\begin{definition}
The \emph{billiard ordered game} joined to the $n$-tuple $(\E_1, \ldots, \E_{n})$ of ellipses from the confocal family \eqref{eq:ConFam}, with \emph{signature} $(i_1, \ldots, i_n) \in \{-1,1\}^n$ is a billiard system such that each trajectory with consecutive vertices $\dots, A_0, A_1, A_2, \dots$ satisfies:
\begin{itemize}
\item $A_{k}\in\E_{s}$ if $s\equiv k\ (\mod\ n)$;
\item at the point $A_{k}$ of the trajectory, the billiard particle is reflected off $\E_{s}$ from inside if $i_s = 1$ and from the outside if $i_s=-1$.
\end{itemize}
We refer to the number $n$ as the \emph{length} of the billiard ordered game.
\end{definition}

\begin{remark}\label{rem:E1}
When the $n$-tuples of ellipses and the signature are simultaneously cyclically permuted, the billiard ordered game remains the same.
Thus, without losing generality, if needed, we can assume that $\E_1$ is the outermost boundary and $\E_n\neq\E_1$.
\end{remark}

\begin{remark}\label{rem:minusone} 
In order for the trajectories of such a game to stay bounded, no two consecutive reflections can be from outside.
Moreover, if the reflection off $\E_{s}$ is from outside, then that ellipse is contained within the previous and the next ellipse in the game. 
\end{remark}

\begin{example}
Figure \ref{fig:OBGame} shows a trajectory corresponding to the billiard ordered game:
$$
(\E_{1}, \E_{3}, \E_{2},\E_{3},\E_{1},\E_{2},\E_{1})
$$
with signature $(1,-1,1,1,1,-1,1)$.
\end{example}

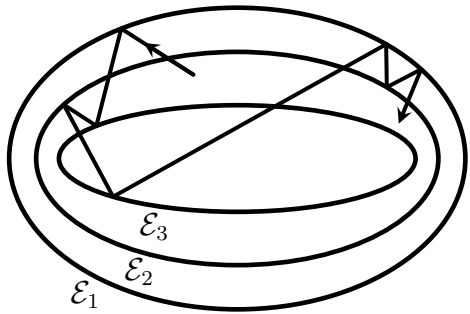
\begin{figure}[htp]
\centering
\begin{tikzpicture}[line cap=round,line join=round,>=stealth,x=1cm,y=1.cm]
\clip(-3.05,-2.05) rectangle (3.05,2.05);
\draw [line width=1.75pt] (0.,0.) ellipse (2.345cm and 0.707cm);
\draw [line width=1.75pt] (0.,0.) ellipse (3.cm and 2.cm);
\draw [line width=1.75pt] (0.,0.) ellipse (2.645cm and 1.414cm);
\draw [line width=1.5pt] (-0.577,1.112)-- (-1.526,1.722);
\draw [line width=1.5pt] (-1.526,1.722)-- (-1.853,0.433);
\draw [line width=1.5pt] (-1.853,0.433)-- (-2.283,0.715);
\draw [line width=1.5pt] (-2.283,0.715)-- (-1.633,-0.507);
\draw [line width=1.5pt] (-1.633,-0.507)-- (1.958,1.515);
\draw [line width=1.5pt] (1.958,1.515)-- (1.967,0.946);
\draw [line width=1.5pt] (1.967,0.946)-- (2.418,1.183);
\draw [->,line width=1.5pt] (-0.577,1.112) -- (-1.230,1.5329);
\draw [->,line width=1.5pt] (2.418,1.183) -- (2.124, 0.462);
%\labels
\draw[color=black] (-1.09,-0.9) node {$\E_{3}$};
\draw[color=black] (-1.29,-1.5) node {$\E_{2}$};
\draw[color=black] (-2,-1.8) node {$\E_{1}$};
\end{tikzpicture}
\caption{A billiard ordered game within three ellipses.}
\label{fig:OBGame}
\end{figure}

One feature of the billiard ordered game is that each trajectory has a fixed caustic from the confocal family \eqref{eq:ConFam}. 
Moreover, the full Poncelet theorem can be applied to such trajectories.
In \cite{DR2004}, the analytic conditions for closure of the trajectories of billiard ordered game, in the Euclidean space of arbitrary dimension, were derived. 

\begin{remark}
Notice that a billiard ordered game can be realised only if the caustic is either an ellipse contained in all ellipses $\E_1$, \dots, $\E_n$ or a hyperbola.
Moreover, it is not clear how to extend the dynamics to all possible caustics.
By constructing the configuration space using billiard books, such dynamics are naturally determined.
\end{remark}

\subsection{Billiard books}\label{sec:books}

Fomenko and his collaborators have produced a set of conjectures and later proved statements which connect the topological structure of various integrable systems and the elliptic billiards in the Euclidean plane. They successfully used latter to model former. 
By considering domains that are subsets of the ellipse which are bounded by arcs of confocal hyperbolas and ellipses, which we call \emph{leaves}, it is conjectured that any integrable system with two degrees of freedom can be represented by gluing copies of such domains together along boundary curves in a suitable fashion. This process constructs a generalised billiard domain called a \emph{billiard book}.  Such gluings are defined by permutations which determine which leaf the billiard moves to upon reflection with the glued boundary. See \cites{VK2018,FV2019,FKK2020} and references therein for details. 

The leaves in the aforementioned sources are often reduced by symmetries across the coordinate axes. 
In contrast to that, in this work, the leaves we use are either elliptic disks or elliptic annuli, which are better suited to the billiard ordered game.
The configuration space we consider will be the union of finitely many such leaves together with gluing permutations along their boundaries.
The precise definition is as follows.

Let $\E_1$, \dots, $\E_n$ be given confocal ellipses and consider a finite collection of \emph{leaves}, such that each of them is either an annulus between two of the given ellipses or an elliptic disk whose boundary is one of those ellipses.
We note that the collection may contain several copies of the same annulus or the same disk, which will be considered as distinct leaves in the book.

The configuration space we consider will be the union of those leaves together with gluing permutations $\sigma_1$, \dots, $\sigma_n$ along their boundaries $\E_1$, \dots, $\E_n$.
More precisely, each $\sigma_k$ is a permutation of the leaves that contain the ellipse $\E_k$ on its boundary.
The motion of the billiard particle is then determined by the following:
\begin{itemize}
\item 
the billiard particle moves along straight segments within each leaf;
	
\item 
after hitting the boundary of a leaf $\L$ at a point of the ellipse $\E_k$, the particle will continue the motion according to the following rules:

\begin{itemize}
	\item[(R1)] on the same leaf $\L$, according to the billiard reflection off $\E_k$, if that ellipse is not the boundary of any of the remaining leaves;
	\item[(R2)] on the leaf $\sigma_k(\mathcal{L})$, according to the billiard reflection law off the common boundary $\E_k$, if all points of both leaves $\L$ and $\sigma_k(\mathcal{L})$ are either within $\E_k$ or outside $\E_k$;
	\item[(R3)] on the leaf $\sigma_k(\mathcal{L})$, along the the same line intersecting the common boundary $\E_k$, if one of the two leaves $\mathcal{L}$, $\sigma_k(\mathcal{L})$ is within $\E_k$ and the other one outside that ellipse.	
\end{itemize}
\end{itemize}
We illustrate the construction and application of those rules in the following example.

\begin{example}\label{ex:1}
	The billiard book consists of three leaves, $\mathcal{L}_1$,  $\mathcal{L}_2$, $\mathcal{L}_3$: the leaf $\mathcal{L}_1$ is the annulus between the ellipses $\E_{1}$ and $\E_{2}$, and 
	the leaves $\mathcal{L}_2$, $\mathcal{L}_3$ are both elliptic disks with the boundary $\E_{2}$. 
	Those leaves have the common boundary $\E_2$ and the gluing permutation is $\sigma_2=(\L_1\L_2\L_3)$.
	For shortness and simplicity of notation, we will use here and in the examples from the following Section \ref{sec:Examples}, the indices of the leaves in the permutation, instead of leaves themselves, so we denote $\sigma_2=(123)$.
	
	The leaves and a segments of sample trajectories are shown in Figure \ref{fig:Ex1}.
	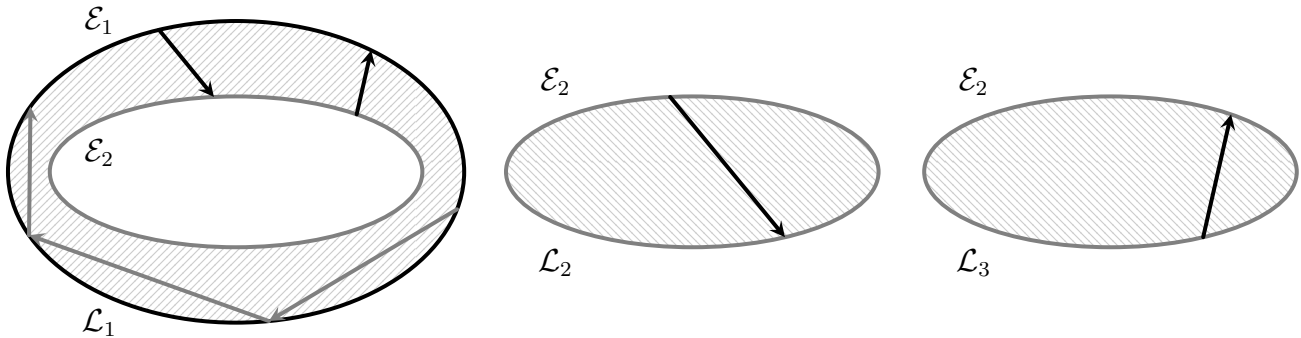
\begin{figure}[htp]
		\centering
		\begin{tikzpicture}[line cap=round,line join=round,>=stealth,x=1.0cm,y=1.0cm]
		%first two ellipses for L_1
		\draw [line width=1.5pt,color=black,pattern=north east lines,pattern color=gray!40] (0.,0.) ellipse (3.cm and 2.cm);
		\draw [line width=1.5pt, 
		%dash pattern=on 3pt off 3pt, 
		color=gray, fill=white, fill opacity=1.0] (0.,0.) ellipse (2.45cm and 1.cm);
		%second ellipse for L_2
		\draw [line width=1.5pt,  color=gray, pattern=north west lines, pattern color=gray!40] (6.,0.) ellipse (2.45cm and 1.cm);
		%Third ellipse for L_3
		\draw [line width=1.5pt, color=gray, pattern=north west lines, pattern color=gray!40] (11.5,0) ellipse (2.45cm and 1.cm);
		%Full piecewise trajectory
		\draw [->,line width=1.5pt] (-1.013,1.883)-- (-0.291,0.993);
		\draw [->,line width=1.5pt] ($(6,0)+(-0.291,0.993)$)-- ($(6,0)+(1.219,-0.867)$); % shift by (6,0) to get to L2
		\draw [->,line width=1.5pt] ($(11.5,0)+(1.219, -0.867)$)-- ($(11.5,0)+(1.586, .762)$); % shift by (11.5,0) to get to L3
		\draw [->,line width=1.5pt] (1.586, .762)-- (1.777, 1.611);
		% Gray trajectory (caustic between E1 and E2):
		\draw [->,line width=1.5pt,gray] (2.906,-0.4967)-- (0.432,-1.979);
		\draw [->,line width=1.5pt,gray] (0.432,-1.979)-- (-2.723,-0.840);
		\draw [->,line width=1.5pt,gray] (-2.723,-0.840)-- (-2.704,0.866);
		%labels
		\draw[color=black] (-1.8,2.0) node {$\E_{1}$};
		\draw[color=black] (-1.8,0.25) node {$\E_{2}$};
		\draw[color=black] (-1.8,-2.0) node {$\mathcal{L}_1$};
		\draw[color=black] (4.2,1.20) node {$\E_{2}$};
		\draw[color=black] (4.2,-1.20) node {$\mathcal{L}_2$};
		\draw[color=black] (9.7,1.20) node {$\E_{2}$};
		\draw[color=black] (9.7,-1.20) node {$\mathcal{L}_3$};
		\end{tikzpicture}
		\caption{The leaves $\mathcal{L}_1, \mathcal{L}_2, \mathcal{L}_3$ with the gluing permutation $\sigma_2 = (123)$ represent the billiard book from Example \ref{ex:1}, together with two sample trajectories. 
		}	\label{fig:Ex1}
	\end{figure}
	
	If the caustic is an ellipse containing $\E_2$, then the corresponding trajectories will be trajectories of billiard within $\E_1$.
	Otherwise, the trajectories are of the billiard ordered game $(\E_1,\E_2)$ with signature $(1,1)$.

The rules for the motion of the particle after hitting the boundary are here applied as follows:
\begin{itemize}
\item if the particle is moving on $\L_1$ and hits its boundary $\E_1$, then according to rule (R1), it continues motion on $\L_1$ according to the billiard reflection law off $\E_1$;
\item if the particle is on $\L_2$ and hits its boundary $\E_2$, then following rule (R2), it will continue motion on the leaf $\L_3=\sigma_2(\L_2)$, according to the billiard reflection law off $\E_2$;
\item if the particle is on $\L_3$ and hits its boundary $\E_2$, then applying rule (R3), it continues its motion on $\L_1=\sigma_2(\L_3)$, along the extension of the same segment. The same rule is applied when the particle moves on $\L_1$ and hits $\E_2$: it will continue its motion on $\L_2=\sigma_2(\L_1)$ along the straight line.
\end{itemize}
\end{example}

\begin{remark}
We note that the configuration space introduced in this section is generally not a surface, since the neighbourhoods of the points on the boundaries of the leaves are not disks.
For the same reason, the trajectories of such systems will not be time reversible.	
That also implies that the billiards on the books are not Hamiltonian.
However, such billiards share a lot of properties of integrable Hamiltonian systems: the phase space will be foliatied into 2-tori and singular level sets and, moreover, the motion on each leaf is determined by Hamiltonian equations.

Furthermore, note that both billiard books and flat classical billiards can be considered as systems on a piece-wise smooth phase space, which are continuous at the preimage of the boundary. Hamiltonian property and integrability thus naturally exist on the smooth parts, and gluing of such parts will be well-defined because of the reflection law and continuity.
The problem of the extension of the smooth and symplectic structures to the neighbourhood of the preimage of the boundary is nontrivial and it is connected with the general procedure of the Hamiltonian gluing introduced by V.~Lazutkin \cite[Chapter I, Section 4.6]{Lazutkin}.
That procedure is based on transversality property and can be applied to a large subclass of classical billiards, which can be extended to a wide subclass of billiard books.
See also a recent paper \cite{Kudr2015}.

We also mention \emph{topological billiards} -- a subclass of billiard books
 introduced by V.~Vedyushkina \cite{Fokicheva2015}, which satisfy the reversibility property, since only two leaves can be glued by each edge.
\end{remark}

In the next section, we analyse in detail more examples of books obtained in the described way.

\section{Billiard ordered games of small length} \label{sec:Examples}

In this section, we will list all billiard order games of length $n\le3$ and construct examples of billiard books which realise those games.
%The topological description of the obtained billiards within books is provided in Section \ref{sec:fomenko}.
We also give examples of billiard books that realize some billiard ordered games of length $4$.

Note that the only billiard ordered game of length $1$ is the billiard within an ellipse.

\subsection{The games of length 2}

For a pair of confocal ellipses $(\E_{1},\E_2)$, such that $\E_2$ is within $\E_1$, there are two games: with signatures $(1,1)$ and $(1,-1)$, see Remarks \ref{rem:E1} and \ref{rem:minusone}.
The game with the signature $(1,-1)$ is the billiard within the annulus between those two ellipses, thus we don't need a book in order to model the corresponding dynamics.

In the game with signature $(1,1)$, the reflections are always from the inside and alternate between $\E_{1}$ and $\E_{2}$. 
A standard billiard desk, which can be embedded into the plane, cannot any more be constructed for such a system.

The first example for a billiard book for the game with signature $(1,1)$ was presented in Example \ref{ex:1}.
Another book that can serve as a configuration space for that billiard order game is given in the following example.

%According to Remark \ref{rem:E1}, $\E_{2}$ is in the interior of $\E_{1}$.
\begin{example}\label{ex:1b}
Consider now the book consisting of four leaves: $\mathcal{L}_1$,  $\mathcal{L}_2$, $\mathcal{L}_3$ as in Example \ref{ex:1}, and an additional leaf $\mathcal{L}_4$, which is an identical copy of $\L_1$, with the gluing permutations $\sigma_1=(14)$, $\sigma_2=(1234)$, see Figure \ref{fig:Ex1b}.
The trajectories on this book are as follows:
\begin{itemize}
\item If the caustic is a confocal ellipse between $\E_1$ and $\E_2$, then the billiard particle reflects only off $\E_1$, and can be only on leaves $\L_1$ and $\L_4$.
Each time when it hits the boundary $\E_1$, the particle switches between those two leaves, according to the (R2) presented in Section \ref{sec:books} and the permutation $\sigma_1=(14)$.
Two consecutive segments of such a trajectory are shown in gray in Figure \ref{fig:Ex1b}.

\item If the caustic is an ellipse contained in $\E_2$ or a hyperbola, then we have two possibilities:
\begin{itemize}
	\item If the particle, when on $\L_1$ travels towards $\E_2$, then the trajectory corresponds to the billiard game $(\E_1,\E_2)$ with signature $(1,1)$.
	The particle will travel from $\L_1$ to $\L_2$ along a straight segment intersecting $\E_2$ according to (R3), obey the billiard reflection law off $\E_2$ when passing from $\L_2$ to $\L_3$ according to (R2), then continue along the same straight segment to $\L_4$ (R3), and return to $\L_1$ after reflecting off $\E_1$ (R2).
	 Four parts of such a trajectory, one in each leaf of the book, are shown in black in Figure \ref{fig:Ex1b};
\item If the particle, when on $\L_1$ travels towards $\E_1$, then the trajectory corresponds to the billiard game $(\E_1,\E_2)$ with signature $(1,-1)$, or, equivalently, to the  the billiard in the annulus between the two ellipses.
Such a trajectory will be only within $\L_1$ and $\L_4$ and the dynamics follows (R2) at the boundaries $\E_1$ and $\E_2$.
Two consecutive segments are shown as dashed lines in Figure \ref{fig:Ex1b}.
\end{itemize}

\end{itemize}
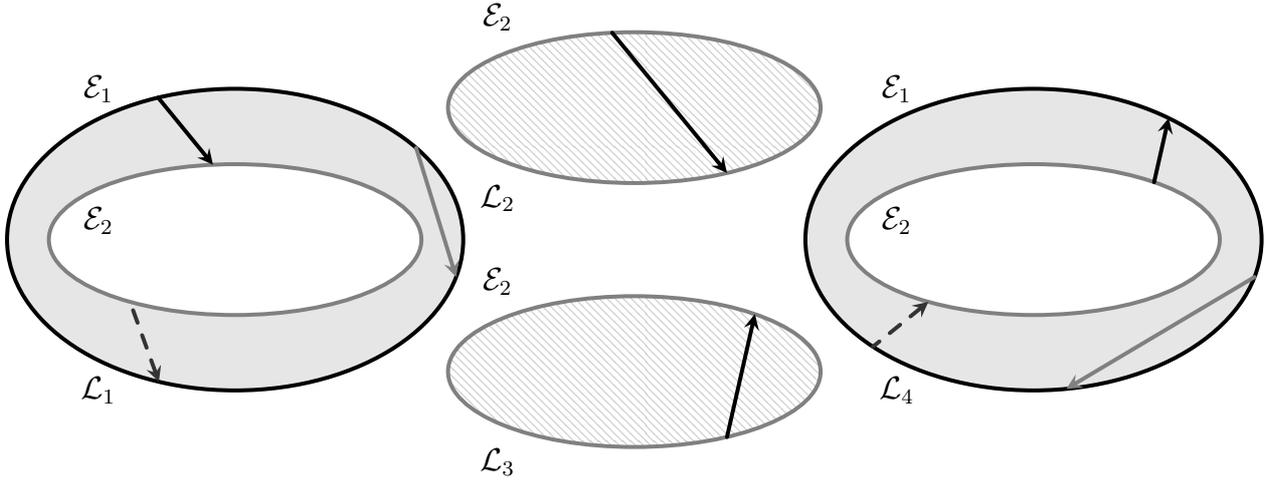
\begin{figure}[htp]
	\centering
	\begin{tikzpicture}[line cap=round,line join=round,>=stealth,x=1.0cm,y=1.0cm]
%first two ellipses for L_1
\draw [line width=1.5pt,color=black,
fill=gray!20,
%pattern=north east lines,pattern color=gray!40
] (0.,0.) ellipse (3.cm and 2.cm);
\draw [line width=1.5pt, 
%dash pattern=on 3pt off 3pt, 
color=gray, fill=white, fill opacity=1.0] (0.,0.) ellipse (2.45cm and 1.cm);
%second ellipse for L_2
\draw [line width=1.5pt, color=gray, pattern=north west lines, pattern color=gray!40] (5.25,1.75) ellipse (2.45cm and 1.cm); 
%Third ellipse for L_3 
\draw [line width=1.5pt, color=gray,pattern=north west lines,pattern color=gray!40] ($(5.25,2.25)+(0.,-4)$) ellipse (2.45cm and 1.cm);
%fourth ellipses for L_4 
\draw [line width=1.5pt,color=black,fill=gray!20
%,pattern=north east lines,pattern color=gray!40
] ($(4.5,4)+(6.,-4)$) ellipse (3.cm and 2.cm);
\draw [line width=1.5pt,  color=gray, fill=white, fill opacity=1.0] ($(4.5,4)+(6,-4)$) ellipse (2.45cm and 1.cm);
%First piecewise full trajectory
\draw [->,line width=1.5pt] (-1.013,1.883)-- (-0.291,0.993);
\draw [->,line width=1.5pt] ($(-0.75,1.75)+(5.709,0.993)$)-- ($(-0.75,1.75)+(7.219,-0.867)$); %
\draw [->,line width=1.5pt] ($(5.25,2.25)+(1.219, -4.867)$)-- ($(5.25,2.25)+(1.586, -3.238)$);
%Second gray piecewise trajectory for L1 and L4 (caustic between E1 and E2)
\draw [->, gray, line width=1.5pt] (2.381,1.216)--(2.906,-0.4967);
\draw [->, gray, line width=1.5pt] ($(10.5,0)+(2.906,-0.4967)$)-- ($(10.5,0)+(0.432,-1.979)$); % shifted by (10.5,0) to get to L4
\draw [->,line width=1.5pt] ($(4.5,4)+(7.586, -3.238)$)-- ($(4.5,4)+(7.777, -2.389)$);
%Third dashed piecewise trajectory for L1 and L4 for billiard in annulus
\draw [<-, black!80, dash pattern=on 5pt off 5pt, line width=1.5pt] (-1.002,-1.885)-- (-1.382,-0.826);
\draw [<-, black!80, dash pattern=on 5pt off 5pt, line width=1.5pt] ($(10.5,0)+(-1.382,-0.826)$)-- ($(10.5,0)+(-2.112,-1.420)$); % shifted by (10.5,0) to get to L4
%Labels
\draw[color=black] (-1.8,2.0) node {$\E_{1}$};
\draw[color=black] (-1.8,0.25) node {$\E_{2}$};
\draw[color=black] (-1.8,-2.0) node {$\mathcal{L}_1$}; %above three for L1
\draw[color=black] ($(-0.75,1.75)+(4.2,1.20) $) node {$\E_{2}$};
\draw[color=black] ($(-0.75,1.75)+(4.2,-1.20)$) node {$\mathcal{L}_2$}; % above two for L2
\draw[color=black] ($(5.25,2.25)+(-1.8, -2.8)$) node {$\E_{2}$};
\draw[color=black] ($(5.25,2.25)+(-1.8, -5.2)$) node {$\mathcal{L}_3$};% above two for L3
\draw[color=black] ($(4.5,4)+(4.2, -2)$) node {$\E_{1}$};
\draw[color=black] ($(4.5,4)+(4.2, -3.75)$) node {$\E_{2}$};
\draw[color=black] ($(4.5,4)+(4.2, -6)$) node {$\mathcal{L}_4$}; % above three for L4
\end{tikzpicture}
	\caption{The leaves $\mathcal{L}_1, \mathcal{L}_2, \mathcal{L}_3, \mathcal{L}_4$ and sample trajectories corresponding to the gluing permutations $\sigma_1=(14)$ and $\sigma_2 = (1234)$. See Example \ref{ex:1b}. }
	\label{fig:Ex1b}
\end{figure}
\end{example}

\begin{remark}
Both Examples \ref{ex:1} and \ref{ex:1b} are constructed with the aim to realize the same billiard ordered game.
It is interesting to note that the obtained billiard books allow different dynamics: the book from Example \ref{ex:1b} allows all billiard games which are present in Example \ref{ex:1}, and, in addition to them, the billiard in the annulus bounded by $\E_1$ and $\E_2$. 
\end{remark}

\subsection{The games of length 3}

In a billiard game of length $3$, either all reflections are from inside or only one of them is from outside.
We will here assume that the ellipses are $\E_1$, $\E_2$, $\E_3$ and that $\E_3$ is within $\E_2$, which is within $\E_1$.

In Examples \ref{ex:111}--\ref{ex:111c} we construct billiard books where the billiard ordered game with signature $(1,1,1)$ will be realised.
In that setting, there are two billiard ordered games: $(\E_1,\E_2,\E_3)$ and $(\E_1,\E_3,\E_2)$.
In Examples \ref{ex:111} and \ref{ex:111c}, the first of them is realised and in Example \ref{ex:111b}, the second one.

In Examples \ref{ex:11-1}--\ref{ex:11-1d} we construct billiard books where the billiard game with one reflection from outside is realised.
According to Remark \ref{rem:minusone}, the reflection from outside then must be off the innermost ellipse $\E_3$.
Thus, there are two possible billiard games: $(\E_1,\E_2,\E_3)$ with signature $(1,1,-1)$ and $(\E_1,\E_3,\E_2)$ with signature $(1,-1,1)$.

\begin{example}\label{ex:111}
The billiard book consists of five leaves:
$\mathcal{L}_1$ is the annulus between $\E_1$ and $\E_2$, $\L_2$ is the elliptic disk within $\E_2$, $\L_3$ is the annulus between $\E_2$ and $\E_3$, and $\L_4$, $\L_5$ are two copies of the disk within $\E_3$, with the gluing permutations $\sigma_2 = (123)$, $\sigma_3 = (345)$, see Figure \ref{fig:Ex111}.
The trajectories on this book can be one of the following:
\begin{itemize}
\item
If the caustic is an ellipse between $\E_1$ and $\E_2$, the particle moves only on $\L_1$ and its trajectories are billiard trajectories within $\E_1$.
Two consecutive segments of such a trajectory are shown in gray within leaf $\L_1$ in Figure \ref{fig:Ex111}.
\item
If the caustic is an ellipse between $\E_2$ and $\E_3$, the particle moves only on $\L_1$, $\L_2$, $\L_3$ and its trajectories are of the billiard ordered game $(\E_1,\E_2)$ with signature $(1,1)$.
A few segments of such a trajectory are shown as dashed lines in Figure \ref{fig:Ex111}.
\item
In all other cases, the trajectories are of the billiard ordered game $(\E_1,\E_2,\E_3)$ with signature $(1,1,1)$.
A few segments of such a trajectory are shown as solid black lines in Figure \ref{fig:Ex111}.
\end{itemize}
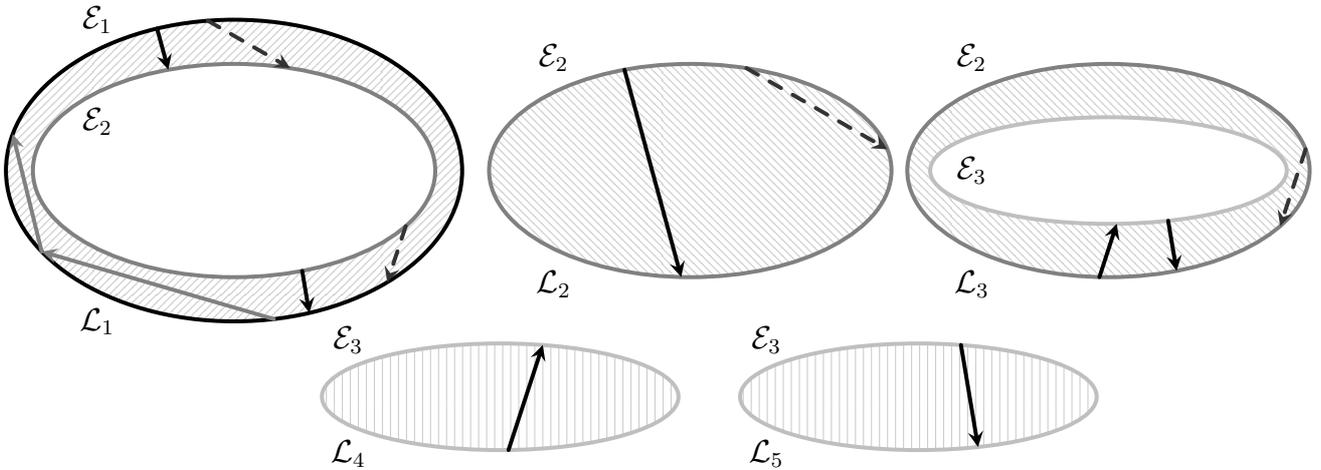
\begin{figure}[htp]
	\begin{tikzpicture}[line cap=round,line join=round,>=stealth,x=1.0cm,y=1.0cm]
%first two ellipses for L_1
\draw [line width=1.5pt,color=black,fill=gray,pattern=north east lines,pattern color=gray!40] (0.,0.) ellipse (3.cm and 2.cm);
\draw [line width=1.5pt, color=gray, fill=white, fill opacity=1.0] (0.,0.) ellipse (2.646cm and 1.414cm);
%second ellipse for L_2
\draw [line width=1.5pt, color=gray, pattern=north west lines, pattern color=gray!40] (6.,0.) ellipse (2.646cm and 1.414cm);
%third ellipses for L_3
\draw [line width=1.5pt, color=gray,pattern=north west lines,pattern color=gray!40] (11.5,0.) ellipse (2.646cm and 1.414cm);
\draw [line width=1.5pt, color=gray!50, fill=white, fill opacity=1.0] (11.5,0.) ellipse (2.345cm and 0.707cm);
%fourth ellipse for L_4
\draw [line width=1.5pt, color=gray!50, pattern=vertical lines, pattern color=gray!40] (3.5,-3) ellipse (2.345cm and 0.707cm);
%fifth ellipse for L_5
\draw [line width=1.5pt, color=gray!50, pattern=vertical lines, pattern color=gray!40] (9,-3) ellipse (2.345cm and 0.707cm);
%Piecewise trajectory segments for full trajectory
\draw [->,line width=1.5pt] (-1.013,1.883)-- (-0.865,1.337);
\draw [->,line width=1.5pt] ($(6,0)+(-0.865,1.337)$)-- ($(6,0)+(-0.119,-1.413)$);% shifted by (6,0) to get to L2
\draw [->,line width=1.5pt] ($(11.5,0)+(-0.119,-1.413)$)-- ($(11.5,0)+(0.109,-0.706)$); % shifted by (11.5,0) to get to L3
\draw [->,line width=1.5pt] ($(3.5,-3)+(0.109,-0.706)$)-- ($(3.5,-3)+(0.560,0.687)$); % shifted by (3.5,-3) to get to L4
\draw [->,line width=1.5pt] ($(9,-3)+(0.560,0.687)$)-- ($(9,-3)+(0.785,-0.666)$); % shifted by (9,-3) to get to L5
\draw [->,line width=1.5pt] ($(11.5,0)+(0.785,-0.666)$)-- ($(11.5,0)+(0.896, -1.331)$); % shifted by (11.5,0) to get to L3
\draw [->,line width=1.5pt] (0.896, -1.331)-- (0.989,-1.888); % back on L1 again
%Second gray Piecewise trajectory for L1 (caustic between E1 and E2)
\draw [->,line width=1.5pt,gray] (0.523,-1.969)-- (-2.530,-1.075);
\draw [->,line width=1.5pt,gray] (-2.530,-1.075)-- (-2.911,0.484);
%Third dashed Piecewise trajectory (caustic between E2 and E3)
\draw [->, black!80, dash pattern=on 5pt off 5pt, line width=1.5pt] (-0.345,1.987)-- (0.738,1.358);
\draw [->, black!80, dash pattern=on 5pt off 5pt, line width=1.5pt] ($(6,0)+(0.738,1.358)$)-- ($(6,0)+(2.593,0.282)$); %shifted by (6,0) to get to L2
\draw [->, black!80, dash pattern=on 5pt off 5pt, line width=1.5pt] ($(11.5,0)+(2.593,0.282)$)-- ($(11.5,0)+(2.258,-0.737)$); % shifted by (11.5,0) to get to L3
\draw [->, black!80, dash pattern=on 5pt off 5pt, line width=1.5pt] (2.258,-0.737)-- (2.014,-1.482);
%Labels
\draw[color=black] (-1.8,2.0) node {$\E_{1}$};
\draw[color=black] (-1.8,0.65) node {$\E_{2}$};
\draw[color=black] (-1.8,-2.0) node {$\mathcal{L}_1$};% previous three for L1
\draw[color=black] ($(6,0)+(-1.8,1.50)$) node {$\E_{2}$};
\draw[color=black] ($(6,0)+(-1.8,-1.50)$) node {$\mathcal{L}_2$};% previous two for L2
\draw[color=black] ($(11.5,0)+(-1.8,1.50)$) node {$\E_{2}$};
\draw[color=black] ($(11.5,0)+(-1.8,0)$) node {$\E_{3}$};
\draw[color=black] ($(11.5,0)+(-1.8,-1.50)$) node {$\mathcal{L}_3$};% previous three for L3
\draw[color=black] ($(3.5,-3)+(-2,0.75)$) node {$\E_{3}$};
\draw[color=black] ($(3.5,-3)+(-2,-0.75)$) node {$\mathcal{L}_4$}; % previous two for L4
\draw[color=black] ($(3.5,-3)+(3.5,0.75)$) node {$\E_{3}$};
\draw[color=black] ($(3.5,-3)+(3.5,-0.75)$) node {$\mathcal{L}_5$};% above labels for L5
\end{tikzpicture}
	\caption{The leaves $\mathcal{L}_1, \ldots, \mathcal{L}_5$ from Example \ref{ex:111} and a sample trajectory corresponding to the gluing permutations $\sigma_2 = (123)$ and $\sigma_3 = (345)$. }
	\label{fig:Ex111}
\end{figure}
\end{example}

\begin{example}\label{ex:111b}
The billiard book has the same leaves $\L_1$, $\L_2$, $\L_3$, $\L_4$, $\L_5$ as in the Example \ref{ex:111} but the gluing permutations are $\sigma_2=(132)$ and $\sigma_3=(354)$.
In this book, if the caustic is an ellipse within $\E_3$ or a hyperbola, the trajectories are of the billiard ordered game $(\E_1,\E_3,\E_2)$ with signature $(1,1,1)$.
An example of such a trajectory is represented by black solid segments in Figure \ref{fig:Ex111c}, but the billiard particle traces it in the opposite direction.
\end{example}

\begin{example}\label{ex:111c}
The billiard book consists of the same $\L_1$, $\L_2$, $\L_3$, $\L_4$, $\L_5$ as in the Example \ref{ex:111}, and the leaf $\L_6$, which is the annulus between $\E_1$ and $\E_3$, with the gluing permutations $\sigma_1 = (16)$, $\sigma_2 = (123)$, $\sigma_3 = (3456)$, see Figure \ref{fig:Ex111c}.
Trajectories on this book can be:
\begin{itemize}
\item
If the caustic is an ellipse between $\E_1$ and $\E_2$, the particle moves only on leaves $\L_1$ and $\L_6$ and the trajectories correspond to the billiard within $\E_1$.
Two consecutive segments of such a trajectory are shown in gray in Figure \ref{fig:Ex111c}.
\item 
If the caustic is an ellipse between $\E_2$ and $\E_3$, the particle moves only on leaves $\L_1$, $\L_2$, $\L_3$, $\L_6$ and the trajectories correspond to the billiard ordered game $(\E_1,\E_1,\E_2)$ with signature $(1,1,1)$.
Dashed segments in Figure \ref{fig:Ex111c} are a part of such a trajectory.
\item
If the caustic is an ellipse within $\E_3$, or a hyperbola, we can have the following:
\begin{itemize}
	\item If the billiard particle on $\L_1$ moves towards $\E_2$, then trajectory will correspond to the billiard ordered game $(\E_1,\E_2,\E_3)$ with signature $(1,1,1)$. Solid black segments in Figure \ref{fig:Ex111c} are a part of such a trajectory.
	\item If the billiard particle on $\L_1$ moves towards $\E_1$, then trajectory will correspond to the billiard ordered game $(\E_1,\E_3)$ with signature $(1,-1)$, i.e.~the billiard in the annulus between $\E_1$ and $\E_3$.
\end{itemize}
\end{itemize}
 \begin{figure}[htp]
	\begin{tikzpicture}[line cap=round,line join=round,>=stealth,x=1.0cm,y=1.0cm]
%first two ellipses for L_1 and L_2
\draw [line width=1.5pt,color=black,fill=gray,pattern=north east lines,pattern color=gray!50] (0.,0.) ellipse (3.cm and 2.cm);
\draw [line width=1.5pt,  color=gray,fill=gray!30] (0.,0.) ellipse (2.646cm and 1.414cm);
%ellipses for L_3 and L_4
\draw [line width=1.5pt, color=gray,fill=gray!30] (6,0.) ellipse (2.646cm and 1.414cm);
\draw [line width=1.5pt, color=gray!70, fill=white, fill opacity=1.0] (6,0.) ellipse (2.345cm and 0.707cm);
\draw [line width=1.5pt, color=gray!70, pattern=vertical lines,pattern color=gray!30] (6,0.) ellipse (2.345cm and 0.707cm);
%ellipses for L_5 and L_6
\draw [line width=1.5pt,color=black,fill=gray,pattern=north east lines,pattern color=gray!50] (12,0) ellipse (3.cm and 2.cm);
\draw [line width=1.5pt, color=gray,fill=gray!30] (12,0.) ellipse (2.345cm and 0.707cm);
%the full piecewise trajectory
\draw [->,line width=1.5pt] (1.013,-1.883)-- (0.865,-1.337);
\draw [->,line width=1.5pt] (0.865,-1.337)-- (0.119,1.413);
\draw [->,line width=1.5pt] ($(6,0)+(0.119,1.413)$)-- ($(6,0)+(-0.109,0.706)$);
\draw [->,line width=1.5pt] ($(6,0)+(-0.109,0.706)$)-- ($(6,0)+(-0.560,-0.687)$);
\draw [->,line width=1.5pt] ($(12,0)+(-0.560,-0.687)$)-- ($(12,0)+(-0.785,0.666)$);
\draw [->,line width=1.5pt] ($(12,0)+(-0.785,0.666)$)-- ($(12,0)+(-0.989,1.888)$);
%Gray Piecewise trajectory for L1 (caustic between E1 and E2)
\draw [->,line width=1.5pt,gray] (-2.151,-1.394)-- (-2.997,-0.090);
\draw [->,line width=1.5pt,gray] ($(12,0)+(-2.997,-0.090)$)-- ($(12,0)+(-2.381,1.217)$);
%Dashed Piecewise trajectory (caustic between E2 and E3)
\draw [->, black!80, dash pattern=on 5pt off 5pt, line width=1.5pt] (-0.345,1.987)-- (0.738,1.358);
\draw [->, black!80, dash pattern=on 5pt off 5pt, line width=1.5pt] (0.738,1.358)-- (2.593,0.282); 
\draw [->, black!80, dash pattern=on 5pt off 5pt, line width=1.5pt] ($(6,0)+(2.593,0.282)$)-- ($(6,0)+(2.258,-0.737)$); % shifted by (6,0) to get to L3
\draw [->, black!80, dash pattern=on 5pt off 5pt, line width=1.5pt] (2.258,-0.737)-- (2.014,-1.482);
\draw [->, black!80, dash pattern=on 5pt off 5pt, line width=1.5pt] ($(12,0)+(2.014,-1.482)$)-- ($(12,0)+(-2.830,-0.663)$); % shifted by (12,0) to get to L6
%labels
\draw[color=black] (-1.8,2.0) node {$\E_{1}$};
\draw[color=black] (-1.8,0.65) node {$\E_{2}$};
\draw[color=black] (0,-2.5) node {$\mathcal{L}_1\text{ and }\mathcal{L}_2$}; % previous three labels for L1 and L2
\draw[color=black] ($(6,0)+(-1.8,1.50)$) node {$\E_{2}$};
\draw[color=black] ($(6,0)+(-1.8,0)$) node  {$\E_{3}$};
\draw[color=black] ($(6,0)+(0,-2.50)$) node {$\mathcal{L}_3\text{ and }\mathcal{L}_4$}; %previous three labels for L3 and L4
\draw[color=black] ($(12,0)+(-1.8,2.0)$) (10.2,2.0) node {$\E_{1}$};
\draw[color=black] ($(12,0)+(-1.8,0)$) node  {$\E_{3}$};
\draw[color=black] ($(12,0)+(0,-2.50)$) node {$\mathcal{L}_5\text{ and }\mathcal{L}_6$};% previous three labels for L5 and L6
\end{tikzpicture}
	\caption{The leaves $\mathcal{L}_1, \ldots, \mathcal{L}_6$ from Example \ref{ex:111c} and sample trajectories corresponding to the gluing permutations $\sigma_1 =(16)$, $\sigma_2 = (123)$, and $\sigma_3 = (3456)$. }
	\label{fig:Ex111c}
\end{figure}
\end{example}

\begin{example}\label{ex:11-1}
The billiard book has three leaves, as shown in Figure \ref{fig:Ex11-1}:
$\L_1$ is the annulus between $\E_1$ and $\E_2$, $\L_2$ is the disc with boundary $\E_2$, and $\L_3$ the annulus between $\E_2$ and $\E_3$, with the gluing permutation $\sigma_2 = (123)$. 
The trajectories on that book are:
\begin{itemize}
\item
If the caustic is an ellipse between $\E_1$ and $\E_2$, the particle moves only on $\L_1$ and the trajectories will correspond to the billiard within $\E_1$.
A part of such a trajectory is shown as gray segments in Figure \ref{fig:Ex11-1}.
\item
If the caustic is an ellipse between $\E_2$ and $\E_3$,  the trajectories correspond to the billiard game $(\E_1,\E_2)$ with signature $(1,1)$, see the example shown as dashed segments in Figure \ref{fig:Ex11-1}.
\item
In all other cases, the trajectories correspond to the game $(\E_{1}, \E_{2},\E_{3})$ with signature $(1,1,-1)$, see the solid black segments in Figure \ref{fig:Ex11-1}.
\end{itemize}
\begin{figure}[htp]
\begin{tikzpicture}[line cap=round,line join=round,>=stealth,x=1.0cm,y=1.0cm]
%first two ellipses for L_1
\draw [line width=1.5pt,color=black,fill=gray,pattern=north east lines,pattern color=gray!40] (0.,0.) ellipse (3.cm and 2.cm);
\draw [line width=1.5pt, color=gray, fill=white, fill opacity=1.0] (0.,0.) ellipse (2.646cm and 1.414cm);
%second ellipse for L_2
\draw [line width=1.5pt,  color=gray, fill=gray!20] (6.,0.) ellipse (2.646cm and 1.414cm);
%third ellipses for L_3
\draw [line width=1.5pt, color=gray,fill=gray!20] (11.5,0.) ellipse (2.646cm and 1.414cm);
\draw [line width=1.5pt, color=gray!50, fill=white, fill opacity=1.0] (11.5,0.) ellipse (2.345cm and 0.707cm);
%First full trajectory 
\draw [->,line width=1.5pt] (-1.345,1.788)-- (-1.051,1.298);
\draw [->,line width=1.5pt]  ($(6,0)+(-1.051,1.298)$)--  ($(6,0)+(0.560,-1.382)$);%shifted by (6,0) to get to L2
\draw [->,line width=1.5pt] ($(11.5,0)+(0.560,-1.382)$)-- ($(11.5,0)+(0.790,-0.666)$); % shifted by (11.5,0) to get to L3
\draw [->,line width=1.5pt] ($(11.5,0)+(0.790,-0.666)$)-- ($(11.5,0)+(1.144,-1.275)$); % shifted by (11.5,0) to get to L3
\draw [->,line width=1.5pt] (1.144,-1.275)-- (1.425,-1.760);
%Second gray Piecewise trajectory for L1 (caustic between E1 and E2)
\draw [->,line width=1.5pt,gray] (0.523,-1.969)-- (-2.530,-1.075);
\draw [->,line width=1.5pt,gray] (-2.530,-1.075)-- (-2.911,0.484);
%Third dashed Piecewise trajectory (caustic between E2 and E3)
\draw [->, black!80, dash pattern=on 5pt off 5pt, line width=1.5pt] (-0.345,1.987)-- (0.738,1.358);
\draw [->, black!80, dash pattern=on 5pt off 5pt, line width=1.5pt] ($(6,0)+(0.738,1.358)$)-- ($(6,0)+(2.593,0.282)$); %shifted by (6,0) to get to L2
\draw [->, black!80, dash pattern=on 5pt off 5pt, line width=1.5pt] ($(11.5,0)+(2.593,0.282)$)-- ($(11.5,0)+(2.258,-0.737)$); % shifted by (11.5,0) to get to L3
\draw [->, black!80, dash pattern=on 5pt off 5pt, line width=1.5pt] (2.258,-0.737)-- (2.014,-1.482);
%labels
\draw[color=black] (-1.8,2.0) node {$\E_{1}$};
\draw[color=black] (-1.8,0.65) node {$\E_{2}$};
\draw[color=black] (-1.8,-2.0) node {$\mathcal{L}_1$}; %previous three for L1
\draw[color=black]  ($(6,0)+(-1.8,1.5)$) node {$\E_{2}$};
\draw[color=black] ($(6,0)+(-1.8,-1.5)$) node {$\mathcal{L}_2$}; %previous two for L2
\draw[color=black] ($(11.5,0)+(-1.8,1.50)$) node {$\E_{2}$};
\draw[color=black] ($(11.5,0)+(-1.8,0)$) node {$\E_{3}$};
\draw[color=black] ($(11.5,0)+(-1.8,-1.50)$) node {$\mathcal{L}_3$}; %previous three for L3
\end{tikzpicture}
\caption{The leaves $\mathcal{L}_1, \mathcal{L}_2, \mathcal{L}_3$ from Example \ref{ex:11-1} and a few trajectories on the book with gluing $\sigma_2=(123)$. }
\label{fig:Ex11-1}
\end{figure}
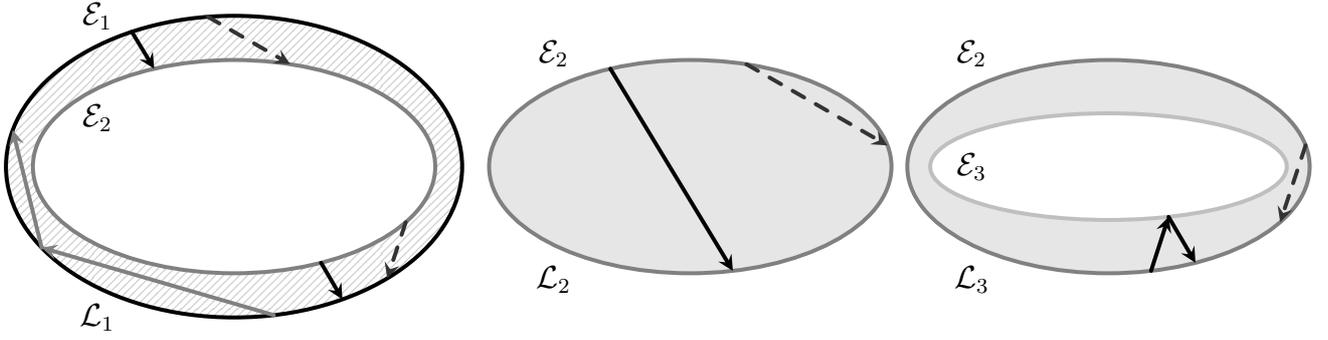
\end{example}

\begin{example}\label{ex:11-1b}
The book has the same leaves as in Example \ref{ex:11-1}, but the gluing permutation is $\sigma_2=(132)$.
In that book, if the caustic is  an ellipse contained within $\E_3$ or a hyperbola, the trajectories correspond to the billiard ordered game $(\E_1,\E_3,\E_2)$ with signature $(1,-1,1)$.
A trajectory is shown in Figure \ref{ex:11-1}, just the billiard particle will trace it in reverse direction.
\end{example}

\begin{example}\label{ex:11-1c}
The billiard book consists of the leaves $\L_1$, $\L_2$, $\L_3$ as defined in Example \ref{ex:11-1} and the fourth leaf $\L_4$, which is the annulus between $\E_1$ and $\E_3$, with the gluing permutations $\sigma_1 = (14)$, $\sigma_2 = (123)$, $\sigma_3 = (34)$, see Figure \ref{fig:Ex11-1c}.
The trajectories are:
\begin{itemize}
	\item 
If the caustic is an ellipse between $\E_1$ and $\E_2$, the particle moves on $\L_1$ and $\L_4$ and the trajectories will correspond to the billiard within $\E_1$, see the gray segments in Figure \ref{fig:Ex11-1c}.
\item
If the caustic is an ellipse between $\E_2$ and $\E_3$, the trajectories correspond to the billiard game $(\E_1,\E_1,\E_2)$ with signature $(1,1,1)$, see the dashed segments in Figure \ref{fig:Ex11-1c}.
\item
For all other caustics, there are two possibilities.
\begin{itemize}
	\item 
In the first one, when the particle is on $\L_1$, it moves towards $\E_2$. 
The trajectories then correspond to the game $(\E_{1}, \E_{2},\E_{3})$ with signature $(1,1,-1)$, as shown by black segments in Figure \ref{fig:Ex11-1c}.
\item
The second possibility is that the particle moves towards $\E_1$, when it is on $\L_1$. In that case, the trajectories do not intersect $\L_2$ and they correspond to the billiard ordered game $(\E_1,\E_3)$ with signature $(1,-1)$, see the dotted segments in Figure \ref{fig:Ex11-1c}.
\end{itemize}
\end{itemize}
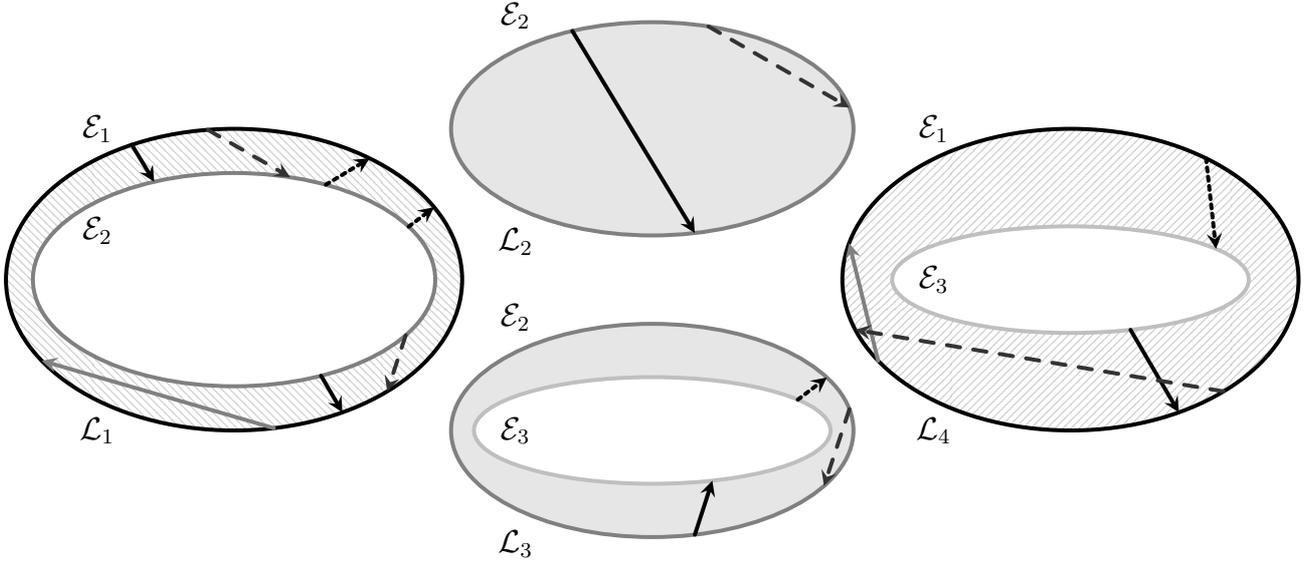
\begin{figure}[htp]
\begin{tikzpicture}[line cap=round,line join=round,>=stealth,x=1.0cm,y=1.0cm]
%first two ellipses for L_1
\draw [line width=1.5pt,color=black,pattern=north west lines,pattern color=gray!40] (0.,0.) ellipse (3.cm and 2.cm);
\draw [line width=1.5pt, color=gray, fill=white, fill opacity=1.0] (0.,0.) ellipse (2.646cm and 1.414cm);
%second ellipse for L_2
\draw [line width=1.5pt, color=gray, fill=gray!20] (5.5,2) ellipse (2.646cm and 1.414cm);
%third ellipses for L_3
\draw [line width=1.5pt, color=gray,fill=gray!20] (5.5,-2) ellipse (2.646cm and 1.414cm);
\draw [line width=1.5pt, color=gray!50, fill=white, fill opacity=1.0] (5.5,-2) ellipse (2.345cm and 0.707cm);
%fourth ellipses for L_4
\draw [line width=1.5pt, color=black,fill=gray,pattern=north east lines,pattern color=gray!40] (11,0) ellipse (3.cm and 2.cm);
\draw [line width=1.5pt, color=gray!50, fill=white, fill opacity=1.0] (11,0) ellipse  (2.345cm and 0.707cm);
%First gray Piecewise trajectory (caustic between E1 and E2)
\draw [->,line width=1.5pt,gray] (0.523,-1.969)-- (-2.530,-1.075);
\draw [->,line width=1.5pt,gray] ($(11,0)+(-2.530,-1.075)$)-- ($(11,0)+(-2.911,0.484)$); % Shift by (11,0) to get to L4
%Second dashed Piecewise trajectory (caustic between E2 and E3)
\draw [->, black!80, dash pattern=on 5pt off 5pt, line width=1.5pt] (-0.345,1.987)-- (0.738,1.358);
\draw [->, black!80, dash pattern=on 5pt off 5pt, line width=1.5pt] ($(5.5,2)+(0.738,1.358)$)-- ($(5.5,2)+(2.593,0.282)$); %shifted by (5.5,2) to get to L2
\draw [->, black!80, dash pattern=on 5pt off 5pt, line width=1.5pt] ($(5.5,-2)+(2.593,0.282)$)-- ($(5.5,-2)+(2.258,-0.737)$); % shifted by (5.5,-2) to get to L3
\draw [->, black!80, dash pattern=on 5pt off 5pt, line width=1.5pt] (2.258,-0.737)-- (2.014,-1.482);
%Third full trajectory starting with E1->E2
\draw [->,line width=1.5pt] (-1.345,1.788)-- (-1.051,1.298);
\draw [->, line width=1.5pt]  ($(5.5,2)+(-1.051,1.298)$)--  ($(5.5,2)+(0.560,-1.382)$);%shifted by (5.5,2) to get to L2
\draw [->,line width=1.5pt] ($(5.5,-2)+(0.560,-1.382)$)-- ($(5.5,-2)+(0.790,-0.666)$); % shifted by (5.5,-2) to get to L3
\draw [->,line width=1.5pt] ($(11,0)+(0.790,-0.666)$)-- ($(11,0)+(1.425,-1.760)$); % shifted by (11.5,0) to get to L4
\draw [->,line width=1.5pt] (1.144,-1.275)-- (1.425,-1.760);
%Fourth full trajectory starting with E2->E1
\draw [->, black, dotted,
%dash pattern=on 5pt off 5pt, 
line width=1.5pt] (1.207,1.259)-- (1.783,1.609);
\draw [->, black, dotted,
% dash pattern=on 5pt off 5pt, 
line width=1.5pt]  ($(11,0)+(1.783,1.609)$)--  ($(11,0)+(1.914,0.408)$);%shifted by (11,0) to get to L4
\draw [->, black, dotted,
%dash pattern=on 5pt off 5pt, 
line width=1.5pt] ($(5.5,-2)+(1.914,0.408)$)-- ($(5.5,-2)+(2.296,0.702)$); % shifted by (5.5,-2) to get to L3
\draw [->, black, dotted,
% dash pattern=on 5pt off 5pt, 
line width=1.5pt] ($(0,0)+(2.296,0.702)$)-- ($(0,0)+(2.631,0.960)$); % shifted by (11.5,0) to get to L4
%Labels
\draw[color=black] (-1.8,2.0) node {$\E_{1}$};
\draw[color=black] (-1.8,0.65) node {$\E_{2}$};
\draw[color=black] (-1.8,-2.0) node {$\mathcal{L}_1$}; % previous three for L1
\draw[color=black] ($(-0.5,2)+(4.2,1.50)$) node {$\E_{2}$};
\draw[color=black] ($(-0.5,2)+(4.2,-1.50)$) node {$\mathcal{L}_2$}; %previous two for L2
\draw[color=black] ($(-6,-2)+(9.7,1.50)$) node {$\E_{2}$};
\draw[color=black] ($(-6,-2)+(9.7,0)$) node {$\E_{3}$};
\draw[color=black] ($(-6,-2)+(9.7,-1.50)$) node {$\mathcal{L}_3$}; % previous three for L3
\draw[color=black] ($(11,0)+(-1.8,2.0)$) node {$\E_{1}$};
\draw[color=black] ($(11,0)+(-1.8,0)$) node {$\E_{3}$};
\draw[color=black] ($(11,0)+(-1.8,-2.)$) node {$\mathcal{L}_4$}; % previous three for L4
\draw [->, black!80, dash pattern=on 5pt off 5pt, line width=1.5pt] ($(11,0)+(2.014,-1.482)$)-- ($(11,0)+(-2.829,-0.665)$); % shifted by (11,0) to get to L4

\end{tikzpicture}
	\caption{The leaves $\mathcal{L}_1, \mathcal{L}_2, \mathcal{L}_3,\mathcal{L}_4$ from Example \ref{ex:11-1c} and a few trajectories with the gluing permutations $\sigma_1 = (14)$, $\sigma_2 = (123)$, $\sigma_3 = (34)$.}\label{fig:Ex11-1c}
\end{figure}
\end{example}

\begin{example}\label{ex:11-1d}
The billiard book consists of the same leaves as in the Example \ref{ex:11-1c}, see Figure \ref{fig:Ex11-1c}, with the gluing permutations $\sigma_1 = (14)$, $\sigma_2 = (132)$, $\sigma_3 = (34)$. 
Then the billiard ordered game $(\E_1,\E_3,\E_2)$ with signature $(1,-1,1)$ will be realised there in the case when the caustic does not contain the smallest ellipse $\E_3$ and the particle, when it is on $\L_1$, moves towards $\E_2$.
Such a trajectory is shown by solid black segments in Figure \ref{fig:Ex11-1c}, but here it is traced in the opposite direction.
In all other cases, the corresponding dynamics will be as in Example \ref{ex:11-1c}.
\end{example}

\subsection{The games of length 4}

In this section, we present a few examples of the billiard books where billiard ordered games of length $4$ are realised.

\begin{example}\label{ex:111-1}
Consider the billiard ordered game $(\E_{1},\E_{2},\E_{1},\E_{3})$ with signature $(1,-1,1,1)$. 
The billiard book has four leaves: $\mathcal{L}_1$ is the annulus between $\E_1$ and $\E_2$, $\L_2$ is the annulus between $\E_1$ and $\E_3$, while $\L_3$, $\L_4$ are two copies of the disk with boundary $\E_3$, with the gluing permutations $\sigma_1 = (12)$ and $\sigma_3 = (234)$, see Figure \ref{fig:Ex111-1}.
 
If the caustic is an ellipse between $\E_1$ and $\E_2$, then the trajectories correspond to the billiard within $\E_1$.
If the caustic is an ellipse between $\E_2$ and $\E_3$, then the trajectories correspond to the billiard game $(\E_1,\E_1,\E_2)$ with the signature $(1,1,-1)$. 
In all other cases, the trajectories correspond to the billiard game $(\E_{1},\E_{2},\E_{1},\E_{3})$ with signature $(1,-1,1,1)$.
\begin{figure}[htp]
	\centering
	\begin{tikzpicture}[line cap=round,line join=round,>=triangle 60,x=1.0cm,y=1.0cm]
	%first two ellipses for L_1
	\draw [line width=1.5pt, color=black,fill=gray!20] (0.,0.) ellipse (3.cm and 2.cm); 
	\draw [line width=1.5pt, color=gray, fill=white, fill opacity=1.0] (0.,0.) ellipse (2.646cm and 1.414cm);
	%Labels for L1
	\draw[color=black] (-1.8,2.0) node {\normalsize$\E_{1}$};
	\draw[color=black] (-1.8,0.65) node {\normalsize$\E_{2}$};
	\draw[color=black] (-1.8,-2.0) node {\normalsize$\mathcal{L}_1$};
	%Second ellipses for L_2
	\draw [line width=1.5pt, color=black,fill=gray!20] (6.25,0.) ellipse (3.cm and 2.cm);
	\draw [line width=1.5pt,  color=gray!50, fill=white, fill opacity=1.0] (6.25,0.) ellipse  (2.345cm and 0.707cm);
	%Labels for L2
	\draw[color=black] ($(6.25,0)+(-1.8,2.0)$) node {\normalsize$\E_{1}$};
	\draw[color=black] ($(6.25,0)+(-1.8,0)$) node {\normalsize$\E_{3}$};
	\draw[color=black] ($(6.25,0)+(-1.8,-2.)$) node {\normalsize$\mathcal{L}_2$};
	%third ellipse for L_3,L_4
	\draw [line width=1.5pt, color=gray!50, fill=gray, pattern=north west lines, pattern color=gray] (12.0,0) ellipse (2.345cm and 0.707cm);
	%Labels for L3, L4
	\draw[color=black] ($(12,0)+(-1.75,0.85)$) node {\normalsize$\E_{3}$};
	\draw[color=black] ($(12,0)+(-1.75,-0.85)$) node {\normalsize$\mathcal{L}_3,\mathcal{L}_4$};
	\end{tikzpicture}
	\caption{The leaves $\mathcal{L}_1, \ldots, \mathcal{L}_4$ corresponding to the gluing permutations $\sigma_1 = (12)$ and $\sigma_3 = (234)$. See Example \ref{ex:111-1}.}
	\label{fig:Ex111-1}
\end{figure}
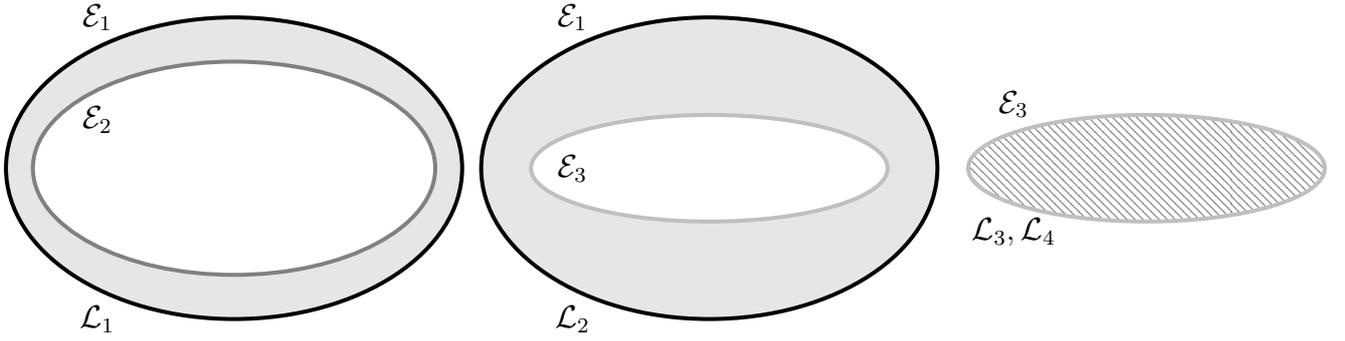

\end{example}

\begin{example}\label{ex:11-1-1}
Consider a billiard ordered game $(\E_{1},\E_{2},\E_{1},\E_{3})$ with signature $(1,-1,1,-1)$. 
The billiard book has two leaves: $\mathcal{L}_1$ is the annulus between $\E_1$ and $\E_2$, while $\L_2$ is the annulus between $\E_1$ and $\E_3$, with the gluing permutation $\sigma_1 = (12)$, see Figure \ref{fig:Ex11-1-1}.
The book constructed in this example is a topological billiard.
\begin{figure}[htp]
\centering
\begin{tikzpicture}[line cap=round,line join=round,>=triangle 60,x=1.0cm,y=1.0cm]
%first two ellipses for L_1
\draw [line width=1.5pt, color=black,fill=gray!20] (0.,0.) ellipse (3.cm and 2.cm);
\draw [line width=1.5pt, color=gray, fill=white, fill opacity=1.0] (0.,0.) ellipse (2.646cm and 1.414cm);
%Second ellipses for L_2
\draw [line width=1.5pt, color=black,fill=gray!20] (8.,0.) ellipse (3.cm and 2.cm);
\draw [line width=1.5pt, color=gray!50, fill=white, fill opacity=1.0] (8.,0.) ellipse  (2.345cm and 0.707cm);
%labels
\begin{scriptsize}
\draw[color=black] (-1.8,2.0) node {\normalsize$\E_{1}$};
\draw[color=black] (-1.8,0.65) node {\normalsize$\E_{2}$};
\draw[color=black] (-1.8,-2.0) node {\normalsize$\mathcal{L}_1$};
\draw[color=black] ($(8,0)+(-1.8,2.0)$) node {\normalsize$\E_{1}$};
\draw[color=black] ($(8,0)+(-1.8,0)$) node {\normalsize$\E_{3}$};
\draw[color=black] ($(8,0)+(-1.8,-2.)$) node {\normalsize$\mathcal{L}_2$};\end{scriptsize}
\end{tikzpicture}
\caption{The leaves $\mathcal{L}_1, \mathcal{L}_2$ corresponding to the gluing permutation $\sigma_1 = (12)$ from Example \ref{ex:11-1-1}. }
\label{fig:Ex11-1-1}
\end{figure}
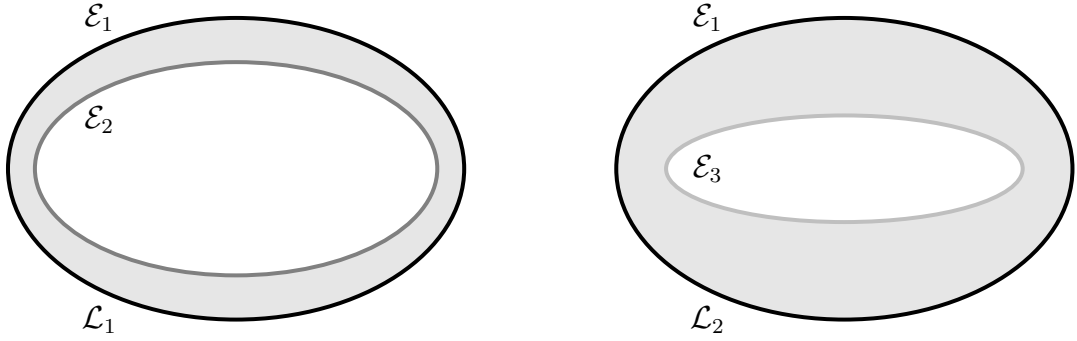

\end{example}

\section{Construction of a book for a given billiard ordered game}\label{sec:construction}

Consider a billiard ordered game 
$(\E_1,\dots,\E_n)$ with signature $(i_1, \ldots, i_n) \in \{-1,1\}^n$. 
Suppose that $\beta_1$, \dots, $\beta_n$ are parameters of those $n$ ellipses, i.e.~$\E_k=\C_{\beta_k}$, for each $k\in\{1,\dots,n\}$ and some values $\beta_1$, \dots, $\beta_n$, which are all smaller than $b$. 
%The parameter $\beta_1$ is the smallest among them, according to Remark \ref{rem:E1}.
%We will also denote $\E_{n+1}=\E_1$ and $\E_0=\E_n$.

We describe a general construction of a billiard book associated to this game. First, we will consider the case when no two consecutive ellipses are equal to each other: $\E_k\neq\E_{k+1}$, for each $k\in\{1,\dots,n\}$.

\begin{theorem}\label{th:algor}
Let $\E_1$, \dots, $\E_n$ be confocal ellipses from the family \eqref{eq:ConFam}, with parameters $\beta_1$, \dots, $\beta_n$.
We assume that no two consecutive parameters are equal, and also $\beta_1\neq\beta_n$.
Denote $\E_{n+1}:=\E_1$, $\E_0:=\E_n$.

Then the billiard ordered game
 $(\E_{1}, \ldots, \E_{n})$ with signature $(i_1, \ldots, i_n)$ is realised on the billiard book consisting of the following leaves:
 \begin{itemize}
 	\item[(A)]
 	the annuli $\A_0$, \dots, $\A_{n-1}$: for each $k\in\{0,\dots,n-1\}$, $\A_k$ is the annulus between ellipses $\E_k$ and $\E_{k+1}$;
 	
 	\item[(D1)] the disks $\D_k$ with boundary $\E_k$, for each $k\in\{1,\dots,n\}$ such that $i_k=1$ and exactly one of the ellipses $\E_{k-1}$, $\E_{k+1}$ is within $\E_k$;

 	\item[(D2)] the disks $\D_k'$, $\D_k''$ with boundary $\E_k$, for each $k\in\{1,\dots,n\}$ such that $i_k=1$ and
 	$\E_k$ is within $\E_{k-1}$ and $\E_{k+1}$.
 	
 \end{itemize}
The gluing permutation for the book along $\E_k$ is:
 \begin{itemize}
 	 	\item $\sigma_k=(\A_{k-1}\D_k\A_k)$ for each $k\in\{1,\dots,n\}$ satifying (D1);
 	 	\item $\sigma_k=(\A_{k-1}\D_k'\D_k''\A_k)$ for each $k\in\{1,\dots,n\}$ satifying (D2);
 	\item $\sigma_k=(\A_{k-1}\A_k)$, otherwise.
 \end{itemize}
\end{theorem}
\begin{proof}
Suppose that the caustic is a an ellipse contained within all boundaries $\E_1$, \dots, $\E_n$ or hyperbola, and that the particle is on the leaf $\A_{k-1}$, moving towards $\E_{k}$.

If $i_k=-1$, then $\E_{k}$ is within $\E_{k-1}$ and $\E_{k+1}$, see Remark \ref{rem:minusone}, and $\sigma_k=(\A_{k-1}\A_k)$.
Note that both annuli $\A_{k-1}$, $\A_k$ are on the same side of their common boundary $\E_k$.
Thus, after hitting $\E_k$ from outside, the particle will continue motion on $\sigma_k(\A_{k-1})=\A_k$,  according to the billiard reflection law off 
$\E_k$, see rule (R2) from Section \ref{sec:books}.
Moreover, on $\A_k$, the particle moves away from $\E_k$ and towards $\E_{k+1}$.

If $i_k=1$ and (D1) is satisfied, then we have two possible cases: either $\E_{k-1}$ is within $\E_k$ and $\E_{k}$ within $\E_{k+1}$, or $\E_{k+1}$ is within $\E_k$ and $\E_{k}$ within $\E_{k-1}$.

In the first case, both $\A_{k-1}$ and $\D_k$ are on the same side of their common boundary $\E_k$.
Thus, after hitting $\E_k$ from inside, the particle will continue motion on $\sigma_k(\A_{k-1})=\D_k$,  according to the billiard reflection law off 
$\E_k$, see rule (R2) from Section \ref{sec:books}.
Since $\D_k$ and $\A_{k}$ are on different sides of their commond boundary $\E_k$, according to rule (R3), the particle will pass along the extension of the straight segment from $\D_k$ to $\sigma_k(\D_k)=\A_{k}$ and on $\A_{k}$ will be moving away from $\E_{k}$ and towards $\E_{k+1}$.
In the second case, one can similarly see that the particle moves along a straght line across $\E_k$ from $\A_{k-1}$ to $\D_k$, reflects off $\E_{k}$ from inside and continues its motion on $\A_{k}$ towards $\E_{k+1}$.

If $i_k=1$ and (D2) is satisfied, then the particle crosses $\E_k$ along a straight segment while passing from $\A_{k-1}$ to $\sigma_k(\A_{k-1})=\D_k'$ according to rule (R3), then reflects from inside $\E_k$ and continues on $\sigma_k(\D_k')=\D_k''$ because of rule (R2), and crosses $\E_{k}$ along the straight segment to continue on $\sigma_k(\D_k'')=\A_k$ according to rule (R3).
On $\A_k$, the particle moves towards the boundary $\E_{k+1}$.

If $i_k=1$ and none of (D1), (D2) is satisfied, then both $\E_{k-1}$, $\E_{k+1}$ are within $\E_k$.
According to rule (R2), the particle will reflect off $\E_k$ from within and continue its motion on $\A_k$, moving there towards $\E_{k+1}$.

We conclude that, in each of the cases, the particle which moves on $\A_{k-1}$ towards $\E_{k}$, will reflect once off $\E_k$: from within if $i_k=1$ and from outside if $i_k=-1$, and along next segment will move on $\A_{k}$ towards $\E_{k+1}$.
By simple induction, we get that the motion which starts on $\A_0$ towards $\E_1$ will correspond to the required billiard ordered game.
\end{proof}

Using Theorem \ref{th:algor} we can get the following estimate on the number of the leaves in our construction:

\begin{corollary}\label{cor:numleaves} 
Assume all as in Theorem \ref{th:algor}.
Suppose that $s$ is the number of all ellipses $\E_k$, $k\in\{1,\dots,n\}$, such that $\E_k$ is within $\E_{k-1}$ and $\E_{k+1}$.

Then number $N$ of leaves in the billiard book constructed in Theorem \ref{th:algor} satisfies: 
$$2n-2s\le N\le 2n.$$
\end{corollary}
\begin{proof}
Note that the number of ellipses $\E_k$, $k\in\{1,\dots,n\}$, such that $\E_{k-1}$ and $\E_{k+1}$ are within $\E_k$ is also equal to $s$.

For each ellipse $\E_k$ which is within exactly one of $\E_{k-1}$, $\E_{k+1}$, there are two leaves with index $k$ in the book: $\A_k$ and $\D_k$. 
There is $n-2s$ such ellipses, thus there are $2(n-2s)$ corresponding leaves in the book.

If both $\E_{k-1}$, $\E_{k+1}$ are within $\E_k$, then there is only one corresponding leaf, $\A_k$, in the book.
Since there are $s$ such ellipses $\E_k$, the total number of corresponding leaves is $s$.

If $\E_k$ is within both $\E_{k-1}$ and $\E_{k+1}$, then there is only one corresponding leaf, $\A_k$, if $i_k=-1$, or three corresponding leaves: $\A_k$, $\D_k$, $\D_k'$, if $i_k=1$.
Since there are $s$ such ellipses $\E_k$, the total number of corresponding leaves is between $s$ and $3s$.

Adding the obtained numbers up, we get the stated inequalities.
\end{proof}

\begin{remark}
	The construction from Theorem \ref{th:algor} does not always give a book with minimal number of leaves that realizes the given billiard ordered game.
	
	For example, for the billiard ordered game $(\E_1,\E_2)$ with signature $(1,1)$, the book from Theorem \ref{th:algor} will be as in Example \ref{ex:1b}, with four leaves, while we constructed in Example \ref{ex:1} a book with three leaves.
	
	For the billiard ordered game $(\E_1,\E_2)$ with signature $(1,-1)$, Theorem \ref{th:algor} will give the book consisting of two leaves: each one will be a copy of the annulus between the two ellipses.
	Since that billiard game is a classical billiard within the annulus, there is a book with only one leaf realizing that game: the leaf is the annulus.
\end{remark}

If there are repeated consecutive ellipses in a billiard game, then the corresponding signs in the signature must be positive, i.e.~all the consecutive reflections off a single ellipse are from inside.
Multiple reflections off one ellipse in the billiard book will be realised with multiple consecutive leaves in the book being identical disks with that ellipse as the boundary.
Using that idea, we can formulate a construction of a billiard book where a general billiard ordered game is realised.

\begin{theorem}\label{th:algor2}
Let $\E_1$, \dots, $\E_n$ be confocal ellipses from the family \eqref{eq:ConFam}, with parameters $\beta_1$, \dots, $\beta_n$.
Without loss of generality, we assume that  $\beta_1\neq\beta_n$.
Denote $\E_{n+1}:=\E_1$, $\E_0:=\E_n$.
	
	Then the billiard ordered game
	$(\E_{1}, \ldots, \E_{n})$ with signature $(i_1, \ldots, i_n)$ is realised on the billiard book consisting of the following leaves:
	\begin{itemize}
		\item[(A)]
	the annulus $\A_k$ between ellipses $\E_k$ and $\E_{k+1}$, for each $k\in\{0,\dots,n-1\}$ when $\E_k$ and $\E_{k+1}$ two distinct ellipses;
		
		\item[(D)]
		$s$ identical copies $\D_{k}^1$, \dots, $\D_{k}^s$ of the disk with boundary $\E_k$, for each $k\in\{1,\dots,n\}$ such that $i_k=1$ and: 
\begin{itemize}
	\item $\E_k=\dots=\E_{k+s-1}$, $\E_{k-1}$ is within $\E_k$, which is within $\E_{k+s}$; or
	\item $\E_k=\dots=\E_{k+s-1}$, $\E_{k+s}$ is within $\E_k$, which is within $\E_{k-1}$;
\end{itemize}
		\item[(D')] 
$s+1$ identical copies $\D_{k}^1$, \dots, $\D_{k}^{s+1}$ of the disk with boundary $\E_k$, for each $k\in\{1,\dots,n\}$ such that $i_k=1$, $\E_k=\dots=\E_{k+s-1}$, $\E_{k}$ is within $\E_{k-1}$ and $\E_{k+s}$;

\item[(D'')] $s-1$ identical copies $\D_{k}^1$, \dots, $\D_{k}^{s-1}$ of the disk with boundary $\E_k$, for each $k\in\{1,\dots,n\}$ such that $i_k=1$, $\E_k=\dots=\E_{k+s-1}$, $\E_{k-1}$ and $\E_{k+1}$ are within $\E_{k}$. \end{itemize}

The gluing permutations for the book are:
	\begin{itemize}
\item $\sigma_k=(\A_{k-1}\D_k^1\dots\D_k^s\A_k)$ for each $k$ and $s$ satisfying (D);
\item $\sigma_k=(\A_{k-1}\D_k^1\dots\D_k^{s+1}\A_k)$ for each $k$ and $s$ satisfying (D');
\item $\sigma_k=(\A_{k-1}\D_k^1\dots\D_k^{s-1}\A_k)$ for each $k$ and $s$ satisfying (D'');

			\item $\sigma_k=(\A_{k-1}\A_k)$, otherwise.
	\end{itemize}
\end{theorem}
\begin{proof}
Similar to the proof of Theorem \ref{th:algor}, with adding additional copies of the disk whenever there are multiple identical consecutive ellipses in the game.
\end{proof}

Let us notice that not all trajectories on the books constructed in Theorem \ref{th:algor} and \ref{th:algor2} are of the initial billiard ordered game.
The behavior of the particle will depend on the caustic, the leaf and the direction of the motion at the initial point.
We notice also that, in general, the motion on a billiard book is not time reversible.
The reason for that is that the gluing permutations are not involutions, if they are not all just transpositions or compositions of disjoint transpositions.

\begin{proposition}
For the billiard books constructed in Theorem \ref{th:algor} and \ref{th:algor2}, a trajectory 
corresponds to the billiard ordered game $(\E_1,\dots,\E_n)$ with signature $(i_1,\dots,i_n)$ if it has the following properties:
\begin{itemize}
\item
it contains a point within the leaf $\A_0$;
\item the direction of motion at that point is towards $\E_1$;
\item the caustic an ellipse contained within all ellipses $\E_1$, \dots, $\E_n$ or a hyperbola.
\end{itemize}	
\end{proposition}

\begin{proposition}
	Suppose that a billiard game $(\E_{1}, \E_{2}, \ldots, \E_{n})$ with signature $(i_1,i_2,\ldots, i_n)$ is realised on the billiard book with the leaves $\mathcal{L}_j$ and gluing permutations $\sigma_s$.
	The inverse billiard game $(\E_1, \E_{n}, \ldots, \E_{2})$ with signature $(i_1, i_n, \ldots, i_2)$ will then be realised on the book with the same leaves $\mathcal{L}_j$ and the inverse gluing permutations $\sigma_s^{-1}$.
	The billiard particles in both books trace the same trajectories, but in reverse. 
\end{proposition}

\section{Topology of billiard ordered games}\label{sec:topology}

Here we give a topological description of the phase space of billiards within books.
In Section \ref{sec:foliation}, we summarize the main properties of the foliation of that phase space, while Section \ref{sec:fomenko} contains details on the examples of billiard books presented in Sections \ref{sec:billiards} and \ref{sec:Examples}.

\subsection{Foliation of the phase space}
\label{sec:foliation}

In all billiard systems where the billiard particle moves without constraints within the domain and obeys billiard reflection law off the boundary, all isoenergy manifolds in the phase space are isomorphic to each other, since trajectories do not depend on the speed of the billiard particle.
Thus, we will consider the isoenergy subspace corresponding to the unit speed of the billiard particle.

If $\mathcal{B}$ is the billiard book, consisting of leaves $\L_j$ with the gluing permutations $\sigma_i$, then the isoenergy manifold is the Cartesian product $\mathcal{B}\times\mathbf{S}^1/\sim$ of the book with the unit circle, with the identification $\sim$, which is induced by the gluing permutations.

Due to the existence of caustics, the phase space of billiards within books defined in Section \ref{sec:books} is foliated into compact level sets.
Each level set is the union of finitely many connected components, which are in general $2$-dimensional tori.
Such components are \emph{non-degenerate}.
The degenerate level sets can appear in one of the following cases:
\begin{itemize}
\item the caustic is degenerate;
\item the caustic coincides with the boundary of some of the leaves of the book $\mathcal{B}$.
\end{itemize}
Note that, in the latter case, the behaviour of the billiard particle is not well-defined after touching the caustic.
Depending on the gluing permutation along that boundary ellipse, sometimes the motion of the billiard particle can be extended, so that the flow in the neighbourhood of the touching point will remain continuous, but sometimes that is not possible.
We note that the latter case can happen only when at least one of the leaves glued along that boundary ellipse is annulus outside the ellipse, i.e.~that boundary is non-convex for that leaf.
In the case when we can extend the flow continuously, that level set will be regular, otherwise its neighbourhood in the isoenergy manifold will have the structure of a Fomenko atom, see \cite{VF2017}.
This is why the Fomenko graphs can be used for the description of the foliations of the isoenergy manifolds of the billiards within books, despite the irregularities of such billiards.
For more details, see \cites{VF2019, FV2019, FV2019b, Ved2019} and references therein.

Recall that the Fomenko graphs classify such foliations up to homeomorphisms of their bases, which are Reeb graphs with Fomenko atoms at the vertices.
Fomenko graphs with numerical marks, i.e.~Fomenko-Zieschang invariants, provide classification up to fiber-wise homeomorphism.
Note also that non-homeomorphic $3$-manifolds, such as $S^3$, $\mathbf{RP}^3$, $S^1\times S^2$, etc., can be equipped with foliations giving the same Fomenko graph $\mathbf{A}-\mathbf{A}$ and the numerical marks will show the difference between them.
At the same time, even most simple manifolds, such as $S^3$, can be foliated in different manners and give different Fomenko graphs.

\subsection{Topological description of isoenergy manifolds}\label{sec:fomenko}

In the following, we provide unmarked Fomenko graphs for each of the examples provided in Sections \ref{sec:books} and \ref{sec:Examples}. 
The horizontal axis below each graph corresponds to the parameter $\lambda$ of the caustic.
The parameters of the ellipses $\E_1$, $\E_2$, \dots are $\beta_1$, $\beta_2$, \dots, according to the notation from the beginning of Section \ref{sec:construction}.
We will also denote by $A_j$, $A_j'$ and $B_j$, $B_j'$ the intersection points of the ellipse $\E_j$ with the coordinate axes: $A_j(\sqrt{a-\beta_j},0)$, $A_j'(-\sqrt{a-\beta_j},0)$,
$B_j(0,\sqrt{b-\beta_j})$,  
$B_j'(0,-\sqrt{b-\beta_j})$.

\begin{proposition}\label{prop:fomenko-ex1}
The isoenergy manifold of the billiard within the book from Example \ref{ex:1} is given by the Fomenko graph in Figure \ref{fig:FomenkoEx1}.
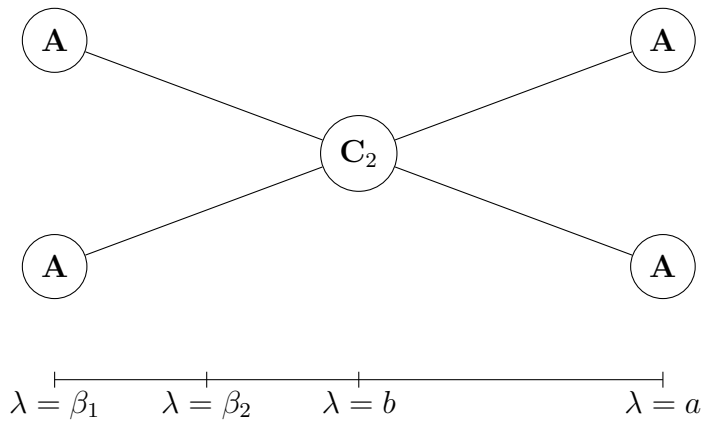
\begin{figure}[htp]
	\centering
	\begin{tikzpicture}[>=stealth',el/.style = {inner sep=3pt, align=left, sloped},em/.style = {inner sep=3pt, pos=0.75, sloped}]
	\tikzset{vertex/.style = {shape=circle,draw,minimum size=1.5em}}
	% vertices
	\node[vertex] (aa) at  (0,1.5) {$\mathbf{A}$};
	\node[vertex] (bb) at  (4,0) {$\mathbf{C}_2$};
	\node[vertex] (ee) at  (8,1.5) {$\mathbf{A}$};
	\node[vertex] (ff) at  (0,-1.5) {$\mathbf{A}$};
	\node[vertex] (jj) at  (8,-1.5) {$\mathbf{A}$};
	%edges
	\path[-] 
	(bb) edge (aa) %top-left edge
	(bb) edge (ee) % top-right edge
	(bb) edge (ff) % bottom-left edge
	(bb) edge (jj); % bottom-right edge
	%axes
	\draw (0,-3) -- (8,-3);
	\foreach \x in {0,2,4,8}
	\draw[shift={(\x,-3)},color=black] (0pt,3pt) -- (0pt,-3pt);
	\node[below] at (0,-3) { $\lambda=\beta_1$};
	\node[below] at (2,-3) { $\lambda=\beta_2$};
	\node[below] at (4,-3) { $\lambda=b$};
	\node[below] at (8,-3) { $\lambda=a$};
	\end{tikzpicture}
	\caption{Proposition \ref{prop:fomenko-ex1}: Fomenko graph for the billiard within the book from Example \ref{ex:1}.}
	\label{fig:FomenkoEx1}
\end{figure}
\end{proposition}
\begin{proof}
The leaves on the level $\lambda=\beta_1$ are degenerate: they correspond to the limit flow along ellipse $\E_1$.
Two singular circles, i.e.~type $\mathbf{A}$ Fomenko atoms correspond to each direction of that flow. 

For each parameter $\lambda\in(\beta_1,\beta_2)\cup(\beta_2,b)$, there are two tori, one corresponding to each direction a trajectory can wind around the ellipse. 

Consider the level set $\lambda=\beta_2$ and notice that the billiard flow can be continuously extended to that level set.
The reason for that is as follows.
Consider the billiard motion parallel to a given tangent line to $\E_2$ and close to the touching point to $\E_2$.
On one side of that tangent line, the particle does not reach $\E_2$, and continues its motion only on leaf $\E_1$, see Figure \ref{fig:Ex1}.
At the other side of the tangent line, the particle crosses $\E_2$ to enter leaf $\L_2$, then bounces off $\E_2$, which occurs when the particle crosses from $\L_2$ to $\L_3$, and finally return again to $\L_1$.
The continuous extension for the motion along the tangent line is simply to carry on straightforward also after touching $\E_2$. This is why the level set for $\lambda=\beta_2$ will also consist of two tori.

For $\lambda=b$, we have a singular level set corresponding to the $\mathbf{C}_2$ Fomenko atom.
On that level set, there are two closed orbits and four separatrices.
One of the closed orbits has reflection points $A_1$ and $A_2'$ and the other one $A_1'$ and $A_2$.
All trajectories on each separatrix are asymptotically approaching one of those closed orbits as the time approaches $+\infty$ and the other one for the time approaching $-\infty$.
On two separatrices, the particle passes through one of the foci moving upwards and through the other one moving downwards, while on the remaining two separatrices that is reversed.

Between $\lambda = b$ and $\lambda = a$, the caustic is a hyperbola and there are again two tori: one torus corresponds to trajectories whose reflections from the inside of $\E_{1}$ are in the upper half-plane and the reflections from $\E_2$ are in the lower half-plane, while the other torus corresponds to trajectories which reflect from $\E_{1}$ in the lower half-plane and from $\E_{2}$ in the upper half-plane.

There are two singular type $\mathbf{A}$ leaves for $\lambda=a$.
One of them is the periodic orbit with reflection points $B_1$ and $B_2'$, the other one the orbit with reflections at $B_1'$ and $B_2$.
\end{proof}

\begin{remark}
The Fomenko graph given in the previous example is the same as that for the billiard in an elliptic annulus, where the inner boundary is confocal to the outer boundary (see Prop. 2.2 of \cite{DR2010}). 
\end{remark}

\begin{proposition}\label{prop:fomenko-ex1b}
	The isoenergy manifold of the billiard within the book from Example \ref{ex:1b} is given by the Fomenko graph in Figure \ref{fig:FomenkoEx1b}.
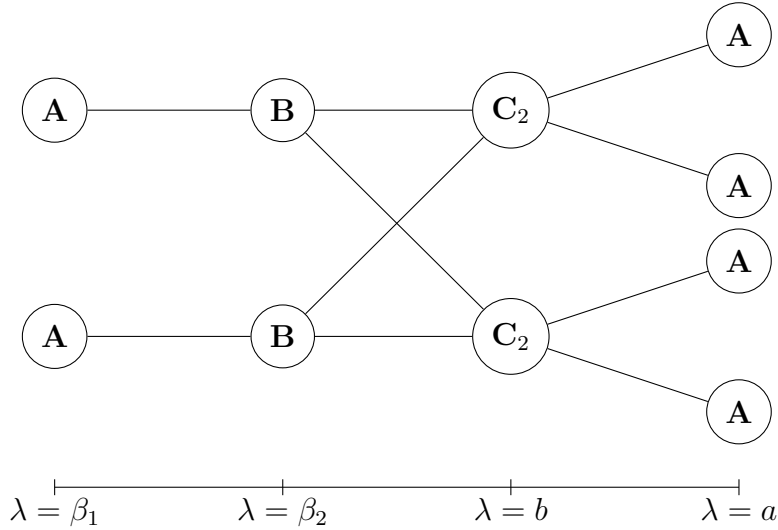
\begin{figure}[htp]
	\centering
	\begin{tikzpicture}[>=stealth',el/.style = {inner sep=3pt, align=left, sloped},em/.style = {inner sep=3pt, pos=0.75, sloped}]
	\tikzset{vertex/.style = {shape=circle,draw,minimum size=1.5em}}
	% vertices
	\node[vertex] (a) at  (0,1.5) {$\mathbf{A}$};
	\node[vertex] (b) at  (0,-1.5) {$\mathbf{A}$};
	\node[vertex] (c) at  (3,1.5) {$\mathbf{B}$};
	\node[vertex] (d) at  (3,-1.5) {$\mathbf{B}$};
	\node[vertex] (e) at  (6,1.5) {$\mathbf{C}_2$};
	\node[vertex] (f) at  (6,-1.5) {$\mathbf{C}_2$};
	\node[vertex] (g) at  (9,2.5) {$\mathbf{A}$};
	\node[vertex] (h) at  (9,0.5) {$\mathbf{A}$};
	\node[vertex] (i) at  (9,-0.5) {$\mathbf{A}$};
	\node[vertex] (j) at  (9,-2.5) {$\mathbf{A}$};
	%edges
	\path[-] 
	(c) edge (a) %top-left edge
	(d) edge (b) % bottom-left edge
	(e) edge (c)
	(f) edge (c)
	(e) edge (d)
	(f) edge (d)
	(e) edge (g)
	(e) edge (h)
	(f) edge (i)
	(f) edge (j); % 
	%axes
	\draw (0,-3.5) -- (9,-3.5);
	\foreach \x in {0,3,6,9}
	\draw[shift={(\x,-3.5)},color=black] (0pt,3pt) -- (0pt,-3pt);
	\node[below] at (0,-3.5) { $\lambda=\beta_1$};
	\node[below] at (3,-3.5) { $\lambda=\beta_2$};
	\node[below] at (6,-3.5) { $\lambda=b$};
	\node[below] at (9,-3.5) { $\lambda=a$};
	\end{tikzpicture}
	\caption{Proposition \ref{prop:fomenko-ex1b}: Fomenko graph for the billiard book from Example \ref{ex:1b}.}
	\label{fig:FomenkoEx1b}
\end{figure}
\end{proposition}
\begin{proof}
The leaves on the levels $\lambda\in[\beta_1,\beta_2)$ 
are described in the same way as in the proof of Proposition \ref{prop:fomenko-ex1}.

On the other hand, in contrast to the situation 
in Proposition \ref{prop:fomenko-ex1}, the billiard flow cannot be continuously extended to the level set $\lambda=\beta_2$.
Namely, consider the billiard particle starting on leaf $\L_1$ (see Figure \ref{fig:Ex1b}), such that its motion is parallel to a given tangent line to $\E_2$ and close to the touching point of that line to $\E_2$.
On one side of that tangent line, the particle does not reach $\E_2$, and continues its motion only on leaf $\E_1$.
On the other side of the tangent line, the particle crosses $\E_2$ to enter leaf $\L_2$, then bounces off $\E_2$, which occurs when the particle crosses from $\L_2$ to $\L_3$, and finally crosses to $\L_4$.
Thus, the continuous extension for the motion along the tangent line would be simply to carry on straightforward also after touching $\E_2$, if we consider the limit from one side, or to cross to $\L_4$, which is the limit from the other side of the tangent line.
Thus, the level set $\lambda=\beta_2$ is singular.
We note that a similar situation was considered in \cite{VF2017}.

That level set still has two connected components, each one corresponding to one direction of motion around the origin. 
Each connected component consists of one circle, which is 1-to-1 projected to the ellipse $\E_2$, and two separatrices, each one projected to the union of the leaves $\L_1$ and $\L_4$.
Thus, those components are isomorphic to the Fomenko $\mathbf{B}$ atoms.

For each parameter $\lambda\in(\beta_2,b)$, there are four tori: two tori correspond to clockwise winding around the caustic, and two to anticlockwise winding.
In both pairs of those tori, one corresponds to the billiard in the annulus between $\E_1$ and $\E_2$, and the other one to the billiard game $(\E_1,\E_2)$ with signature $(1,1)$.

In Figure \ref{fig:FomenkoEx1b}, all three edges joining one $\mathbf{B}$ vertex will correspond to the winding in the same direction around the origin.

For $\lambda=b$, we have two singular level sets, both corresponding to the $\mathbf{C}_2$ Fomenko atom.
In one of those level sets, the two periodic orbits are those with reflections at points $A_1$, $A_2$ and $A_1'$, $A_2'$, while each separatrix contains only orbits always remaining on one side of the $x$-axis, asymptotically approaching one of the periodic orbits as $t\to+\infty$ and the other one as $t\to-\infty$.
That singular level set corresponds to the billiard in the annulus between $\E_1$ and $\E_2$ and all four edges joining the corresponding $\mathbf{C}_2$ vertex in Figure \ref{fig:FomenkoEx1b} also correspond to the billiard motion in the annulus.

In the other $\mathbf{C}_2$ singular level set, the two periodic orbits are those with reflections at points $A_1$, $A_2'$ and $A_1'$, $A_2$.
That singular level set corresponds to the billiard ordered game $(\E_1,\E_2)$ with signature $(1,1)$ and all four edges joining the corresponding $\mathbf{C}_2$ vertex in Figure \ref{fig:FomenkoEx1b} also correspond to the same game.
 
Between $\lambda = b$ and $\lambda = a$, the caustic is a hyperbola and there are again four tori: two tori correspond to trajectories of the billiard in the annulus between $\E_1$ and $\E_2$, one in the upper half-plane, the other in the lower half-plane, while two remaining tori are as described in the proof of Proposition \ref{prop:fomenko-ex1}.
	
There are four singular type $\mathbf{A}$ leaves for $\lambda=a$: they are periodic orbits with reflection points: $B_1$ and $B_2$, $B_1'$ and $B_2'$, $B_1$ and $B_2'$, $B_1'$ and $B_2$.
\end{proof}

\begin{proposition}\label{prop:fomenko-ex111}
	The isoenergy manifolds of the billiard within the books from Examples \ref{ex:111}, \ref{ex:111b}, \ref{ex:11-1}, \ref{ex:11-1b} are given by the Fomenko graph in Figure \ref{fig:FomenkoEx111}.
	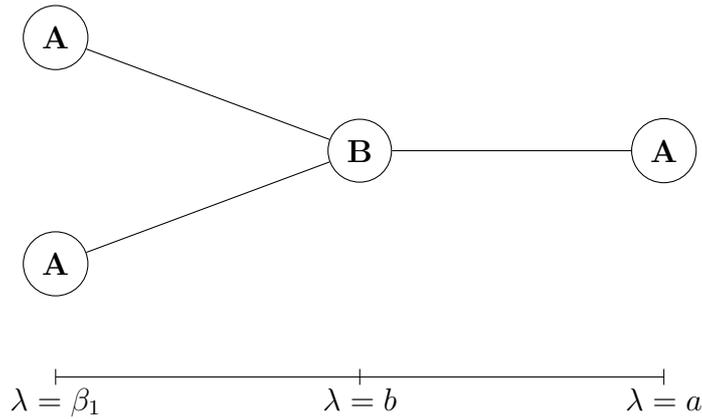
\begin{figure}[h]
		\centering
		\begin{tikzpicture}[>=stealth',el/.style = {inner sep=3pt, align=left, sloped},em/.style = {inner sep=3pt, pos=0.75, sloped}]
		\tikzset{vertex/.style = {shape=circle,draw,minimum size=1.5em}}
		% vertices
		\node[vertex] (aa) at  (0,1.5) {$\mathbf{A}$};
		\node[vertex] (bb) at  (4,0) {$\mathbf{B}$};
		\node[vertex] (ee) at  (8,0) {$\mathbf{A}$};
		\node[vertex] (ff) at  (0,-1.5) {$\mathbf{A}$};
		%edges
		\path[-] 
		(bb) edge (aa) %top-left edge
		(bb) edge (ee) % top-right edge
		(bb) edge (ff); % bottom-left edge
		%axes
		\draw (0,-3) -- (8,-3);
		\foreach \x in {0,4,8}
		\draw[shift={(\x,-3)},color=black] (0pt,3pt) -- (0pt,-3pt);
		\node[below] at (0,-3) { $\lambda=\beta_1$};
		\node[below] at (4,-3) { $\lambda=b$};
		\node[below] at (8,-3) { $\lambda=a$};
		\end{tikzpicture}
		\caption{Proposition \ref{prop:fomenko-ex111}: Fomenko graph for the billiard within books from Examples \ref{ex:111}, \ref{ex:111b} \ref{ex:11-1}, \ref{ex:11-1b}.}
		\label{fig:FomenkoEx111}
	\end{figure}
\end{proposition}
\begin{proof}
The discussion is similar as in Proposition \ref{prop:fomenko-ex1}.
Notice that here we will have only one periodic orbit on the singular level sets $\lambda=b$ and $\lambda=a$.
Each of those orbits becomes closed after six reflections.	
\end{proof}

\begin{proposition}\label{prop:FomenkoEx111c}
The isoenergy manfolds for the billiard book given in Examples \ref{ex:111c}, \ref{ex:11-1c}, \ref{ex:11-1d} are given by the Fomenko graph in Figure \ref{fig:FomenkoEx111c}.
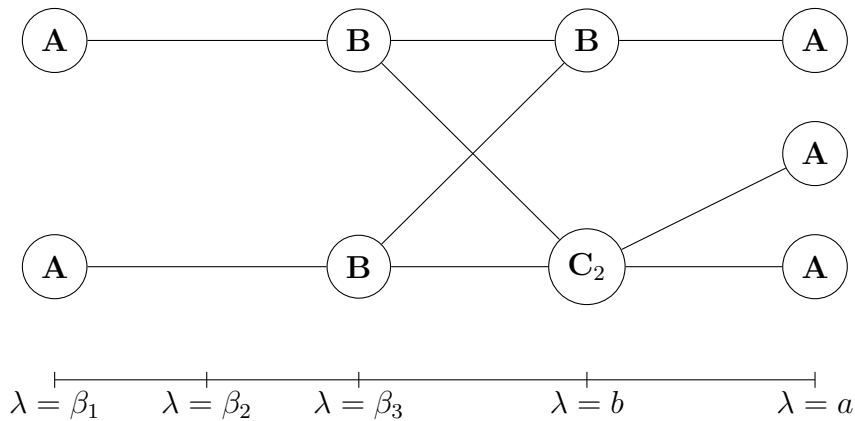
\begin{figure}[h]
	\centering
	\begin{tikzpicture}[>=stealth',el/.style = {inner sep=3pt, align=left, sloped},em/.style = {inner sep=3pt, pos=0.75, sloped}]
	\tikzset{vertex/.style = {shape=circle,draw,minimum size=1.5em}}
	% vertices
	\node[vertex] (a) at  (0,1.5) {$\mathbf{A}$};
	\node[vertex] (b) at  (0,-1.5) {$\mathbf{A}$};
	\node[vertex] (c) at  (4,1.5) {$\mathbf{B}$};
	\node[vertex] (d) at  (4,-1.5) {$\mathbf{B}$};
	\node[vertex] (e) at  (7,1.5) {$\mathbf{B}$};
	\node[vertex] (f) at  (7,-1.5) {$\mathbf{C}_2$};
	\node[vertex] (g) at  (10,1.5) {$\mathbf{A}$};
	\node[vertex] (h) at  (10,0) {$\mathbf{A}$};
	\node[vertex] (i) at  (10,-1.5) {$\mathbf{A}$};
	%edges
	\path[-] 
	(c) edge (a) %top-left edge
	(d) edge (b) % bottom-left edge
	(e) edge (c)
	(f) edge (c)
	(e) edge (d)
	(f) edge (d)
	(e) edge (g)
	(f) edge (h)
	(f) edge (i); % 
	%axes
	\draw (0,-3) -- (10,-3);
	\foreach \x in {0,2,4,7,10}
	\draw[shift={(\x,-3)},color=black] (0pt,3pt) -- (0pt,-3pt);
	\node[below] at (0,-3) { $\lambda=\beta_1$};
	\node[below] at (2,-3) { $\lambda=\beta_2$};
	\node[below] at (4,-3) { $\lambda=\beta_3$};
	\node[below] at (7,-3) { $\lambda=b$};
	\node[below] at (10,-3) { $\lambda=a$};
	\end{tikzpicture}
	\caption{Proposition \ref{prop:FomenkoEx111c}: Fomenko graph for the billiard books given in Examples \ref{ex:111c}, \ref{ex:11-1c}, \ref{ex:11-1d}.}
	\label{fig:FomenkoEx111c}
\end{figure}
\end{proposition}
\begin{proof}
We provide the proof for Example \ref{ex:111c}.
The discussion of the remaining two examples is very similar.

%We describe the geometry of the trajectories corresponding to the tori in terms of the caustic parameter $\lambda$. 

At $\lambda =\beta_1$, there are singular circles: one corresponds to winding around the ellipse $\E_{1}$ in the clockwise direction and the other in the anticlockwise direction. This corresponds to the $\mathbf{A}$-atoms at $\lambda=\beta_1$. 

For $\lambda \in (\beta_1,\beta_3)$, each torus corresponds to clockwise or anticlockwise motion around the ellipse.

At $\lambda = \beta_3$, we have singular level sets corresponding to the $\mathbf{B}$-atoms.
Each of those sets contains one circle, projected to the ellipse $\E_3$, and two separatrices.
All edges joining one $\mathbf{B}$-vertex at the level $\lambda=\beta_3$ correspond to the same direction of winding about the origin.

For $\lambda \in (\beta_3,b)$, two edges contain the tori corresponding to the billiard within the annulus between $\E_1$ and $\E_3$: those edges have joint $\mathbf{C}_2$ vertex at the level $\lambda=b$, while the remaining two edges correspond to the billiard game $(\E_1,\E_2,\E_3)$ with signature $(1,1,1)$ and they have joint $\mathbf{B}$ vertex at level $\lambda=b$.

The description of the edge connecting $\mathbf{B}$ and $\mathbf{A}$ vertices for $\lambda\in(b,a)$, including those two singular level sets is similar as in Proposition \ref{prop:fomenko-ex111}, while the description of the edges connecting $\mathbf{C}_2$ with $\mathbf{A}$ vertices is similar to the corresponding description in Proposition \ref{prop:fomenko-ex1b}.
\end{proof}

\begin{remark}
We note that a general billiard book will have several different regimes of motion, which correspond to different billiard ordered games.
On the corresponding Fomenko graph, each edge is divided to segments which correspond to one billiard ordered game.
For example, the two parts of the edges over segments $(\beta_1, \beta_2]$ of the graph in Figure \ref{fig:FomenkoEx1} correspond to the classical billiard within $\E_1$, while the parts over segments $(\beta_2,b)$ and edges over $(b,a)$ correspond to the billiard ordered game $(\E_1,\E_2)$ with the signature $(1,1)$.
\end{remark}

\section{Conclusions and discussion}
In this work, we united two wide modern generalisations of classical billiards and we beleive that our results may serve as a starting point for further research.

Some questions for consideration may be:
\begin{itemize}
\item to determine all billiard books that realize a given billiard ordered game and, in particular, those books with minimal number of leaves;
\item in particular, it would be interesting to know if a given billiard ordered game can be realised in a book that has less leaves than by construction in Theorem \ref{th:algor} or whose number of leaves is smaller than the bound given in Corollary \ref{cor:numleaves};
\item to determine all possible Fomenko graphs and Fomenko-Zieschang invarinant that may appear in realizations of the billiard ordered games;
\item to determine equivalent systems in the classes of integrable rigid body, their analogs for other Lie algebras, integrable geodesic flows, integrable billiards.	
\end{itemize}

We would also like to mention a parallel to another direction of our research, where we studied \emph{pseudo-integrable billiards}, which are classical billiards within desks bounded by arcs of confocal conics, where corners may appear on the boundary, see \cite{DR2014,DR2014b,DR2015,DR2015b}.
There is an immediate similarity in the settings of those billiards with the billiard ordered games, since in both cases the boundary consists of multiple confocal conics.
Due to the existence of caustics, the phase spaces of both billiard ordered games and pseudo-integrable billiards is foliated into $2$-dimensional leaves.
Those leaves, in both cases, are obtained by gluing several pieces bounded by confocal arcs.
On the other hand, one important distinction in the structure of the phase space is that non-singular leaves for the billiard ordered games are tori, while for pseudo-integrable billiards they can be surfaces of any genus.
The dynamics is also different, since non-convex corners on the boundary imply the existence of separatrices on non-singular level sets.

\section*{Acknowledgements}
We are grateful to anonymous referees for their careful reading of this work and suggestions which enabled us to improve it.

This research is partially supported by the Discovery Project No.~DP200100210 \emph{Geometric analysis of nonlinear systems} from the Australian Research Council, 
by  Mathematical Institute of the Serbian Academy of Sciences and Arts, the Science Fund of Serbia grant Integrability and Extremal Problems in Mechanics, Geometry and
Combinatorics, MEGIC, Grant No.~7744592 and the Ministry for Education, Science, and Technological Development of Serbia and the Simons Foundation grant no.~854861.

\bibliographystyle{amsalpha}
\nocite{*}
\bibliography{References}

\end{document}